\documentclass[
a4paper,
12pt,
DIV=14,`1234
toc=bibliography,
headsepline,
parskip=half-,
abstract=true,
]{scrartcl}

\usepackage{amsmath,amsfonts,amsthm,amssymb,nicefrac}
\usepackage{graphicx,color,url}
\usepackage{xcolor}
\usepackage{accents}
\usepackage{url,hyperref}
\usepackage{enumerate}
\usepackage{comment}
\usepackage{scrlayer-scrpage}

\newtheorem{theorem}{Theorem}
\newtheorem{lemma}[theorem]{Lemma}
\newtheorem{corollary}[theorem]{Corollary}
\newtheorem{proposition}[theorem]{Proposition}
\theoremstyle{definition}
\newtheorem{remark}[theorem]{Remark}

\numberwithin{equation}{section}\numberwithin{theorem}{section}
\allowdisplaybreaks
\newcounter{stepctr}
{\end{list}}

\def\XXint#1#2#3{{\setbox0=\hbox{$#1{#2#3}{\int}$}
 \vcenter{\hbox{$#2#3$}}\kern-.5\wd0}}

\newcommand{\circo}{\accentset{\circ}}
\newcommand{\ra}{\rangle}
\newcommand{\la}{\langle}
\newcommand{\e}{\varepsilon}

\providecommand{\titlemacro}{{Mean Curvature Flow in the Sphere}}
\title{\titlemacro}
\author{Artemis A. Vogiatzi}
\ohead{\titlemacro}
\ihead{A.A.Vogiatzi}
\setkomafont{pageheadfoot}{\footnotesize}

\date{}

\begin{document}
\title{Sharp Quartic Pinching for the Mean Curvature Flow in the Sphere}
\maketitle
\begin{abstract}
We introduce a new sharp quartic curvature pinching for submanifolds in $\mathbb{S}^{n+m}$, $m\ge2$, which is preserved by the mean curvature flow. Using a blow up argument, we prove a codimension and a cylindrical estimate, where in regions of high curvature, the submanifold becomes approximately codimension one, quantitatively, and is weakly convex and moves by translation or is a self shrinker. With a decay estimate, the rescaling converges smoothly to a totally geodesic limit in infinite time, without using Stampacchia iteration or integral analysis.
\end{abstract}
\section{Introduction}\label{sec_introduction}
Let $F_0\colon\mathcal{M}^n \rightarrow \mathcal{N}^{n+m}$, $m\ge2$ be a smooth immersion of a compact manifold $\mathcal{M}^n$. The mean curvature flow starting from $F_0$ is the following family of submanifolds $F\colon \mathcal{M}^n \times[0, T) \rightarrow \mathcal{N}^{n+m},$ such that
	\begin{align*}
	\begin{split}
\left\{
	\begin{array}{rl}
		\partial_t F(p, t) &=H(p, t), \ \ \text{for} \ \ p \in \mathcal{M}, t \in[0, T)   \\
		F(p, 0) &=F_0(p),
	\end{array}
	\right.
\end{split}
	\end{align*}
where $H(p, t)$ denotes the mean curvature vector of $\mathcal{M}_t=F(\mathcal{M}, t)$. This is a system of quasilinear weakly parabolic partial differential equations for $F$. Geometrically, the mean curvature flow is the steepest descent flow for the area functional of a submanifold and hence it is a natural curvature flow.
\\
In the case of codimension one, a crucial step in the study of singularity formation in the mean convex mean curvature flow is the convexity estimate. This states that in regions of large mean curvature, the second fundamental form is almost positive definite.
\\
In \cite{Hu84}, Huisken proved that closed convex hypersurfaces under the mean curvature flow evolve into spherical singularities, using Stampacchia iteration and the Michael-Simon-Sobolev inequality. In \cite{Hu86}, Huisken generalises this theorem to Riemannian background curvature spaces with strict convexity depending on the background curvature. 
\\
In contrast, for mean convex hypersurfaces a structure theory (including convexity estimates and the description of high-curvature regions) was developed by Huisken-Sinestrari and White using different methods: Huisken-Sinestrari proved quantitative convexity (and cylindrical) estimates and later developed surgery for two-convex hypersurfaces, while White’s approach uses compactness and tools from geometric measure theory together with strong maximum principle rigidity for the second fundamental form \cite{HuSi99,HuSi09,Wh03,white2005local}. Haslhofer–Kleiner \cite{HK2} developed an alternative approach to these results based on Andrews’s non-collapsing result for the mean curvature flow \cite{An12}.
\\
The analysis of high codimension mean curvature flow is far more challenging, than in the hypersurface case, since convexity does not make sense anymore. An alternative condition to convexity was introduced by Andrews--Baker in \cite{AnBa10}:  on a compact submanifold, if $|H|>0$, there exists a $c>0$, such that
	\begin{align*}
	|A|^2\le c|H|^2,
	\end{align*}
which is preserved by codimension one mean curvature flow. Also, this condition makes sense for all codimensions. For $c = \min\{\frac{4}{3n},\frac{1}{n-1}\}$, they were able to prove convergence to a round sphere. Note that $|A|^2 - \frac{1}{n-1}|H|^2 < 0, H >0$ implies convexity in codimension one. 
\\
In \cite{Hu87}, Huisken proved a convergence theorem for the mean curvature flow of quadratically pinched hypersurfaces in the sphere $\mathbb{S}^{n+1}$. In \cite{BakerNguyen}, Baker and Nguyen generalised that result and proved a sharp convergence theorem for high codimension mean curvature flow as follows.
\begin{theorem} Let $\mathcal{M}_0=F_0\left(\mathcal{M}^n\right)$ be a closed submanifold smoothly immersed in $\mathbb{S}^{n+m}\left(\frac{1}{\sqrt{\bar{K}}}\right)$, with constant curvature $\bar{K}$. If $\mathcal{M}_0$ satisfies
	\begin{align*}
\left\{\begin{array}{l}
|A|^2 \leq \frac{4}{3 n}|H|^2+\frac{n}{2} \bar{K}, \quad n=2,3 \\
|A|^2 \leq \frac{1}{n-1}|H|^2+2 \bar{K}, \quad n \geq 4,
\end{array}\right.
	\end{align*}
then either the mean curvature flow has a unique, smooth solution on a finite time interval and the submanifold $\mathcal{M}_t$ contracts to a point, as $t\to T$ or it exists for all time and the submanifold $\mathcal{M}_t$ converges to a totally geodesic submanifold of $\mathbb{S}^{n+m}\left(\frac{1}{\sqrt{\bar{K}}}\right)$, as $t\to\infty$.
\end{theorem}
If the relationship of $|A|^2$ and $|H|^2$ is nonlinear, these pinching conditions could be improved. In \cite{LeiXu}, Lei and Xu obtained a convergence theorem for the mean curvature flow of arbitrary codimension in spheres  with some pinching.
In \cite{DongPu}, Pu proved a convergence theorem that improved the theorem of Baker and Nguyen, using a quartic pinching condition and the condition introduced by Lei and Xu in \cite{LeiXu}. 

The purpose of this paper is to introduce a new curvature condition, which is fourth order in the norm of the second fundamental form.
\begin{theorem}
Let $F_0: \mathcal{M} \rightarrow \mathbb{S}^{n+m}\left(\frac{1}{\sqrt{\bar{K}}}\right)$ be an n-dimensional, smooth compact submanifold, $n\ge8$, $m\ge2$, with constant curvature $\bar{K}$ and $T$ be the maximal time of existence. If $\mathcal{M}_0$ satisfies
	\begin{align}\label{quarticpinching}
|A|^2 \leq \sqrt{\left(\frac{|H|^2}{n-2}+4 \bar{K}\right)^2+(4 n-16) \bar{K}^2},
	\end{align}
then one of the following holds. If $T< \infty$, then the blow up of $\mathcal{M}_t$ is a codimension one ancient flow as $t \to T$ or if $T = \infty$, then $\mathcal{M}_t$ converges to a smooth totally geodesic submanifold as $t \to T$.
\end{theorem}

Condition \eqref{quarticpinching} is called the quartic pinching condition for $\mathcal{M}$. 
\begin{remark}
It is well known that theorems of this form, lead to topological classifications, see \cite{HuSi09}, \cite{HTNsurgery}, tori of the form $\mathbb{S}^1\times\mathbb{S}^{n-1}$ or connected sums of tori. A pinching condition, such as the above quartic pinching condition is necessary for such a classification. If the quartic condition does not hold, then there are submanifolds of the sphere that are not tori of the form $\mathbb{S}^1\times\mathbb{S}^{n-1}$ or connected sums of tori. Consider the submanifold $\mathcal{M}=\mathbb{S}^2(r)\times\mathbb{S}^{n-2}(s)\subset\mathbb{S}^{n+1}(1)\subset\mathbb{S}^{n+m}(1)$, with $r^2+s^2=1$, then  $|A|^4-\left(\frac{|H|^2}{n-2}+4\right)^2-(4n-16)=\frac{4((n-2)^2-1)}{(n-2)^2}\cdot\frac{s^4}{r^4}$. Also, it satisfies $\sqrt{\left(\frac{|H|^2}{n-2}+4 \bar{K}\right)^2+(4 n-16) \bar{K}^2}>\frac{|H|^2}{n-2}+4\bar{K}$. The latter indicates that the term involving $\bar{K}$ becomes much more influential when considered in the quartic pinching condition than in the quadratic pinching condition.
\end{remark}

For the case of finite time, the statement is proven by a blow up argument using gradient estimates on the second fundamental form and a cylindrical type estimate. For the case of infinite time, the statement is proven directly using the maximum principle, without using Stampacchia iteration or integral analysis that is used in \cite{AnBa10}, \cite{BakerThesis}, \cite{Hu87}, \cite{LeiXu} and \cite{DongPu}.
\\
Nguyen in \cite{HTNsurgery} developed a surgery construction allowing high codimension mean curvature flow with cylindrical pinching to pass through singularities. This generalised the codimension one result of \cite{HuSi09} (see also \cite{HK2}) to high codimension. A key aspect of this surgery procedure is the codimension estimate presented in \cite{Naff}, which shows that near regions of high curvature, singularities become approximately codimension one. Another crucial component is the  cylindrical estimate, which shows that nears regions of high curvature, the submanifold either becomes strictly positively curved or cylindrical of the form $\mathbb{S}^{n-1}\times\mathbb{R}$.  These estimates are essential for the surgery to work and allow us to control the geometry of the submanifold in regions of high curvature.

This paper is part of a series of papers, aiming to extend Nguyen's result on high codimension mean curvature flow with surgery in Euclidean space to Riemannian manifolds.

The outline of the paper is as follows. In section 2, we give all the technical tools needed for our work and set up our notation. In section 3, we give the proof for the preservation of the quartic pinching condition along the mean curvature flow. In section 4, we prove the gradient estimates. In section 5, we prove the codimension estimate, the cylindrical estimate and give a full classification for the singularity models of quartically pinched solutions in high codimension mean curvature flow. Finally, in section 6, we prove  that the submanifold converges smoothly to a totally geodesic limit in infinite time.
\\
\textbf{Acknowledgements.}  
The author would like to thank her PhD supervisor, Dr Huy The Nguyen, for many invaluable discussions and for his guidance in this work.
\section{Preliminaries}\label{sec_Preliminaries}
This section presents the necessary preliminary results and establishes our notation. 
 We adopt the following convention for indices:
	\begin{align*}
	1\le i,j,k,\ldots\le n, \ 1\le a,b,c,\ldots\le n+m \ \ \text{ and} \ \ 1\le \alpha,\beta,\gamma,\ldots \le m.
	\end{align*}
For any $p \in \mathcal{M}$, denote by $N_p \mathcal{M}$ the normal space of $\mathcal{M}$ at point $p$, which is the orthogonal complement of $T_p \mathcal{M}$ in $F^* T_{F(p)} \mathcal{N}$. 
We choose a local orthonormal frame $\{e_i\}$ of $T_p \mathcal{M}_t$ and $\{\nu_\alpha\}$ for the normal bundle $N_p \mathcal{M}_t$, where
	\begin{align*}
	\nu_1:=\frac{H}{|H|}
	\end{align*} 
is the principal normal direction. This is well defined, since this is used for section 5, where $|H| \neq 0$. \\
We denote by $A$ to be the normal vector valued second fundamental form tensor and denote by $H$ the mean curvature vector which is the trace of the second fundamental form given by $H^\alpha=\operatorname{tr}_g A^\alpha=\sum_i A^\alpha_{ii}$. The tracefree second fundamental form $\mathring{A}$  is defined by $\mathring{A}=A-\frac{1}{n}Hg$, whose components are given by $\mathring{A}^\alpha_{ij}=A^\alpha_{ij}-\frac{1}{n}H^\alpha g_{ij}$. Obviously, we have $\sum_i \mathring{A}^\alpha_{ii}=0$. In particular, the squared length satisfies 
	\begin{align*}
	|\mathring{A}|^2=|A|^2-\frac{1}{n}|H|^2.
	\end{align*}
We denote by $A^-$ the second fundamental form tensor orthogonal to the principal direction and $h$ to be the tensor valued second fundamental form in the principal direction, that is 
	\begin{align*}
	h_{ij}=\frac{\langle A_{ij},H\rangle}{|H|}.
	\end{align*}
Therefore, we have $A=A^-+h\nu_1$ and $|A|^2=|A^-|^2+|h|^2$. Also, $A^+_{ij}=\langle A_{ij},\nu_1\rangle\nu_1$. We denote $\mathring{h}$ to be the traceless part of the second fundamental form in the principal direction.
Let $R^{\perp}$ denote the normal curvature tensor. 
Given a connection $\nabla$ on $A$, from the definition of $A^-$, it is natural to define the connection $\hat{\nabla}^\bot$ acting on $A^-$, by
	\begin{align*}
	\hat{\nabla}^\bot_i A^-_{jk}:=\nabla^\bot_i A^-_{jk}-\langle \nabla^\bot_i A^-_{jk},\nu_1\rangle\nu_1.
	\end{align*}
For the sake of simplicity, we use $\nabla$ instead of $\nabla^\bot$. The traceless second fundamental form can be rewritten as $\mathring{A}=\sum_{\alpha} \mathring{A}^\alpha \nu_{\alpha}$, where
		\begin{align*}
	\begin{cases} H=\operatorname{tr} A=\operatorname{tr}h_{ij},\ \ \ \ \alpha= 1 &\\ H^\alpha=\operatorname{tr} A^\alpha=0, \ \ \ \ \ \ \ \alpha \geq 2\end{cases}
\ \ \text{and} \ \ \ \
	\begin{cases}\mathring{h}_{ij}=h_{ij}-\frac{|H|}{n} Id, \ \ \ \ \alpha= 1 &\\ \mathring{A}^-=A^\alpha,\ \ \ \ \ \ \ \ \ \ \ \ \ \ \ \alpha \geq 2.\end{cases}
	\end{align*}
When using a basis of this kind, we adopt the following notation:
	\begin{align*}
	 |A^-|^2=|\mathring{A}^-|^2:=\sum_{\alpha\ge2}|\mathring{A}^{\alpha}|^2.
	\end{align*}
Let $R_{ijkl}=g\left(R(e_i,e_j)e_k,e_l\right), \bar{R}_{abcd}=\langle\bar{R}(e_a,e_b)e_c,e_d\rangle$ and $R_{ij\alpha\beta}^\bot =\langle R^\bot(e_i,e_j)\nu_\alpha,\nu_\beta\rangle.$\\
In higher codimension, the Ricci equation can be written as follows.
	\begin{align}\label{Riccieq}
R^\bot_{ij\alpha\beta}=\sum_{k}\left(A^\alpha_{ip}A^\beta_{jp}-A^\beta_{ip}A^\alpha_{jp}\right).
	\end{align}
\begin{proposition}[\cite{AnBa10}, Section 3]\label{evoleqsij}
With the summation convention, the evolution equations of $A_{ij}$ and $H$ are
	\begin{align}\label{eqn_A}
	&(\partial_t - \Delta) A_{ij}=\sum_{p,q,\beta} A_{ij}^{\beta}A_{pq}^{\beta} A_{pq} +\sum_{p,q,\beta}A_{iq}^{\beta}A_{qp}^{\beta} A_{pj}+\sum_{p,q,\beta}A_{jq}^{\beta}A_{qp}^{\beta} A_{pi} - 2 \sum_{p,q,\beta}A_{ip}^{\beta}A_{jq}^{\beta}A_{pq}\nonumber\\
&\quad\quad\quad\quad\quad+ 2 \bar K H g_{ij} - n \bar K A_{ij},\\
\label{eqn_H}		
&\left(\partial_t-\Delta\right) H=\sum_{p,q,\beta} H^{\beta}A_{pq}^{\beta} A_{pq} + n\bar K H.
	\end{align}
\end{proposition}
\begin{lemma}[\cite{BakerThesis}, Section 5.1]\label{evoleqs}
Let us consider a family of immersions $F\colon \mathcal{M}^n\times [0,T) \to \mathbb{S}^{n+m}$ moving by mean curvature flow. Then, we have the following evolution equations
	\begin{align}
&(\partial_t -\Delta) |A|^2= - 2 |\nabla A|^2 + 2 | \langle A, A\rangle |^2 + 2 |R^{\perp}|^2 + 4 \bar K |H|^2 - 2 n \bar K |A|^2, \label{eqn_|A|^2}\\
&(\partial_t - \Delta) |H|^2= - 2 |\nabla H|^2 + 2 |\langle A, H\rangle |^2 + 2 n \bar K |H|^2. \label{eqn_|H|^2}
	\end{align}
\end{lemma}
By Berger's inequality in \cite{Goldberg}, we have
	\begin{align}\label{Berger}
	\begin{split}
&|\bar{R}_{acbc}|\le\frac{1}{2}(K_1+K_2), \ \ \text{ for} \ a\neq b,\\
	&|\bar{R}_{abcd}|\le\frac{2}{3}(K_1+K_2), \ \ \text{for all distinct indices} \ a,b,c,d.
	\end{split}
	\end{align}
\begin{lemma}[\cite{Liu}, Lemma 3.1]\label{katoinequality} For any $\eta>0$ we have the following inequality.
	\begin{align*}
	|\nabla^\perp A|^2 \geq\left(\frac{3}{n+2}-\eta\right)&|\nabla^\perp H|^2-\frac{2}{n+2}\left(\frac{2}{n+2} \eta^{-1}-\frac{n}{n-1}\right)|w|^2,
	\end{align*}
where $w=\sum_{i, j, \alpha} \bar{R}_{\alpha j i j} e_i \otimes \omega_\alpha$.
	\end{lemma}

\begin{proposition}[\cite{Naff}, Proposition 2.2]\label{prop2.2naff}
	\begin{align}\label{2.22naff}
&|\nabla^{\perp} A|^2 =\sum_{i,j,k}|\hat{\nabla}_i^{\perp} A^-_{j k}+h_{j k} \nabla_i^{\perp} \nu_1|^2+\sum_{i,j,k}|\langle\nabla_i^{\perp} A^-_{j k}, \nu_1\rangle+\nabla_i h_{j k}|^2,\\
	\label{2.23naff}	&|\nabla^{\perp} H|^2 =|H|^2|\nabla^{\perp} \nu_1|^2+|\nabla| H||^2,\\
	\label{2.24naff}	&|\nabla^{\perp} A^-|^2 =|\hat{\nabla}^{\perp} A^-|^2+|\langle\nabla^{\perp} A^-, \nu_1\rangle|^2.
	\end{align}
\end{proposition}
For the orthogonal decomposition of $N \mathcal{M}$, we write $N \mathcal{M}=E_1 \oplus \hat{E}$, where $\hat{E}$ consists of normal vectors $\nu$, which are everywhere orthogonal to $\nu_1,\left\langle\nu, \nu_1\right\rangle=0$, and $E_1=C^{\infty}(M) \nu_1$. It is very useful to consider the implications of the Codazzi equation for the decomposition of $\nabla_i^{\perp} A_{j k}$ above. Projecting the Codazzi equation onto $E_1$ and $\hat{E}$ implies the both of the tensors
	\begin{align*}
\nabla_i h_{j k}+\langle\nabla_i^{\perp} A^-_{j k}, \nu_1\rangle \quad \text { and } \quad \hat{\nabla}_i^{\perp} A^-_{j k}+h_{j k} \nabla_i^{\perp} \nu_1
	\end{align*}
are symmetric in $i, j, k$. Consequently, it is equivalent to trace over $j, k$ or trace over $i, k$, and this implies
	\begin{align}\label{2.25naff}
& \sum_{k=1}^n (\nabla_k h_{i k}+\langle\nabla_k^{\perp} A^-_{i k}, \nu_1\rangle)=\nabla_i|H|, \\
	\label{2.26naff}
& \sum_{k=1}^n( \hat{\nabla}_k^{\perp} A^-_{i k}+h_{i k} \nabla_k^{\perp} \nu_1)=|H| \nabla_i^{\perp} \nu_1.
	\end{align}
Following Naff in \cite{Naff}, we have the following decomposition. Recalling that $A^-_{j k}$ is traceless, it is straightforward to verify that
	\begin{align}\label{eq4.15}
	&\sum_{i,j,k}|\nabla_i h_{j k}+\langle\nabla_i^{\perp} A^-_{j k}, \nu_1\rangle|^2 =\sum_{i,j,k}|\nabla_i \mathring{h}_{j k}+\langle\nabla_i^{\perp} A^-_{j k}, \nu_1\rangle|^2+\frac{1}{n}|\nabla| H \|^2,\\
	\label{eq4.17}
	&\sum_{i,j,k}|\langle\nabla_i^{\perp} \mathring{A}_{j k}, \nu_1\rangle|^2=\sum_{i,j,k}|\nabla_i \mathring{h}_{j k}+\langle\nabla_i^{\perp} A^-_{j k}, \nu_1\rangle|^2,\\
	\label{eq4.16}
	&\sum_{i,j,k}|\hat{\nabla}_i^{\perp} A^-_{j k}+h_{j k} \nabla_i^{\perp} \nu_1|^2 =\sum_{i,j,k}|\hat{\nabla}_i^{\perp} A^-_{j k}+\mathring{h}_{j k} \nabla_i^{\perp} \nu_1|^2+\frac{1}{n}|H|^2|\nabla^{\perp} \nu_1|^2.
	\end{align}
As in Lemma \ref{katoinequality}, we obtain that
	\begin{align}
\label{4.20naff}
&\frac{3}{n+2}|H|^2|\nabla^\bot \nu_1|^2\le\sum_{i,j,k}|\hat{\nabla}^\bot_i A^-_{jk}+h_{jk}\nabla^\bot_i \nu_1|^2.
	\end{align}
Now expanding the right-hand side of both inequalities above using \eqref{eq4.15} and \eqref{eq4.16}, recalling \eqref{eq4.17} and noting that $\frac{3}{n+2}-\frac{1}{n}=\frac{2(n-1)}{n(n+2)}$, we arrive at the estimates
	\begin{align}\label{4.21naff}
	&\frac{2(n-1)}{n(n+2)}|\nabla| H \|^2 \leq\sum_{i,j,k}|\langle\nabla_i^{\perp} \mathring{A}_{j k}, \nu_1\rangle|^2, \\
	\label{4.22naff} 
	&\frac{2(n-1)}{n(n+2)}|H|^2|\nabla^{\perp} \nu_1|^2 \leq\sum_{i,j,k}|\hat{\nabla}_i^{\perp} A^-_{j k}+\mathring{h}_{j k} \nabla_i^{\perp} \nu_1|^2.
	\end{align}
We follow the notation from \cite{HuSi09}. Given $p \in \mathcal{M}$ and $r>0$, we let $\mathcal{B}_{g(t)}(p, r) \subset \mathcal{M}$ be the closed ball of radius $r$ around $p$ with respect to the metric $g(t)$. In addition, if $t, \theta$ are such that $0 \leq t-\theta<t \leq T$, we set
	\begin{align}\label{parneigh}
&\mathcal{P}(p, t, r, \theta)=\left\{(q, s): q \in \mathcal{B}_{g(t)}(p, r), s \in[t-\theta, t]\right\}.
	\end{align}
Such a set will be called a (backward) parabolic neighbourhood of $(p, t)$.
\begin{lemma}[\cite{AnBa10}, \cite{BakerThesis}]
	\begin{align}\label{eqn_|A|^2traceless}
&\left(\partial_t-\Delta\right)|\mathring{A}|^2= -2|\nabla \mathring{A}|^2+2 R_1-\frac{2}{n} R_2-2 n \bar{K}|\mathring{A}|^2,
	\end{align}
where
\begin{align*}
	R_1&=\sum_{\alpha,\beta}\left(\sum_{i,j} A^\alpha_{ij} A_{ij}^\beta \right)^2+\sum_{i,j,\alpha,\beta}\left(\sum_p\left(A^\alpha_{ip}A_{jp}^\beta -A^\alpha_{jp}A_{ip}^\beta \right)\right)^2,\\
	R_2&=\sum_{i,j}\left(\sum_{\alpha} H^\alpha A^\alpha_{ij}\right)^2=|h|^2|H|^2=|A|^2|H|^2-P_2|H|^2,\\
	P_2&=\sum_{\alpha>1}\left(\mathring{h}_{i j}^{ \alpha}\right)^2=|\mathring{A}^-|^2=|A^-|^2.
	\end{align*}
	\end{lemma}
\section{Preservation of the quartic pinching condition}
In this section, we show the preservation of the quartic pinching under the mean curvature flow. Denote
	\begin{align}\label{a(x)}
a(x):=\sqrt{\left(\frac{x}{n-2}+4 \bar{K}\right)^2+(4n-16) \bar{K}^2}, \quad \mathring{a}(x)=a(x)-\frac{x}{n}.
	\end{align}
By direct computations, we get
	\begin{align}\label{ineqa}
	&\frac{x}{n-2}+4\bar{K}<a(x)<\frac{x}{n-2}+2\sqrt{n}\bar{K},\\
&\label{a'(x)}\mathring{a}^{\prime}(x)=\frac{\frac{x}{n-2}+4\bar{K}}{(n-2) \sqrt{\left(\frac{x}{n-2}+4 \bar{K}\right)^2+(4 n-16) \bar{K}^2}}-\frac{1}{n}<\frac{2}{n(n-2)}, \\
\label{a''(x)} &\mathring{a}^{\prime \prime}(x)=\frac{(4n-16) \bar{K}^2}{(n-2)^2\left(\sqrt{\left(\frac{x}{n-2}+4 \bar{K}\right)^2+(4n-16) \bar{K}^2}\right)^3} .
	\end{align}
\begin{lemma}[cf. \cite{DongPu}]\label{inequalities}
For $x \geq 0$ and $n\ge5, \mathring{a}$ has the following properties.
	\begin{itemize}
\item [(i)]$\frac{4 x\left(\mathring{a}^\prime \right)^2}{\mathring{a}}<1,$
\item [(ii)]$2 x \mathring{a}^{\prime\prime}+\mathring{a}^\prime<\frac{2(n-1)}{n(n+2)},$
\item [(iii)]$n \bar{K}\left(\mathring{a}+x \mathring{a}^{\prime}\right)-a\left(\mathring{a}-x \mathring{a}^{\prime}\right) \geq \frac{8n(n-4)(n-2)^2\bar{K}^4}{(x+\sqrt{8+2\sqrt{4n}}(n-2)\bar{K})^2},$ 
\item [(iv)]$2 \mathring{a}-\frac{x}{n}+x \mathring{a}^{\prime} \leq 2\sqrt{4n}\bar{K}-\frac{(n-8)x}{n(n-2)}-\frac{6(\sqrt{4n}-4)x\bar{K}}{3x+2\sqrt{4n}(n-2)\bar{K}},$ 
\item [(v)]$\frac{x}{n-2}(a+n \bar{K})-\left(\frac{x}{n-2}+4 n \sqrt{n}\bar{K}\right)\left(\mathring{a}+a-n \bar{K}-x \mathring{a}^{\prime}\right)\\
<-\frac{2 x \bar{K}}{n-2}\left (2n\sqrt{n}-n+2\right)+4 n\sqrt{n} (n-8) \bar{K}^2.$
	\end{itemize}
\end{lemma}
\begin{proof}
i) From \eqref{a'(x)}, we have
	\begin{align*}
	\frac{4 x\left(\mathring{a}^\prime \right)^2}{\mathring{a}}<\frac{2x}{n(n-2)\mathring{a}}\frac{8}{n(n-2)}<1.
	\end{align*}
ii) Using \eqref{a'(x)}, \eqref{a''(x)} and Young's inequality, we have
	\begin{align*}
	&2 x \mathring{a}^{\prime\prime}+\mathring{a}^\prime=\frac{4x(4n-16)\bar{K}^2}{(n-2)^2\sqrt{\left(\frac{x}{n-2}+4\bar{K}\right)^2+(4n-16)\bar{K}^2}^3}\\
	&<\frac{2x(4n-16)\bar{K}^2}{(n-2)^2\left(\frac{x}{n-2}+4\bar{K}\right)\left(\left(\frac{x}{n-2}+4\bar{K}\right)^2+(4n-16)\bar{K}^2\right)}\\
        	&=\frac{2(4n-16)\bar{K}^2}{\frac{x^2}{n-2}+12x\bar{K}+\frac{16n(n-2)^2\bar{K}^3}{x}+4(n+8)(n-2)\bar{K}^2}\\
	&<\frac{4(2n-8)\bar{K}^2}{\frac{x^2}{n-2}+4\sqrt{48n}(n-2)\bar{K}^2+4(n+8)(n-2)\bar{K}^2}\\
	&<\frac{4(2n-8)\bar{K}^2}{4\sqrt{48n}(n-2)\bar{K}^2+4(n+8)(n-2)\bar{K}^2}\\
	&=\frac{2(n-4)}{(\sqrt{48n}+n+8)(n-2)}\\
	&<\frac{2(n-4)}{(n-2)(n+2)}.
	\end{align*}
Therefore,
	\begin{align*}
	2x\mathring{a}^{\prime\prime}+\mathring{a}^\prime<\frac{2(n-4)}{(n-2)(n+2)}+\frac{2}{n(n-2)}=\frac{2(n-1)}{n(n+2)}.
	\end{align*}
iii) By a direct computation, we have that
	\begin{align*}
	-a(\mathring{a}-x\mathring{a}^\prime)&=-\left(\frac{x}{n-2}+4\bar{K}\right)^2-(4n-16)\bar{K}^2+\frac{\frac{x^2}{n-2}+4x\bar{K}}{n-2}=-\frac{4x\bar{K}}{n-2}-4n\bar{K}^2.
	\end{align*}
Also, 
	\begin{align*}
	&n\bar{K}(\mathring{a}+x\mathring{a}^\prime)=n\bar{K}\left(\frac{(n-2)\left(\left(\frac{x}{n-2}+4\bar{K}\right)^2+(4n-16)\bar{K}^2\right)+\frac{x^2}{n-2}+4x\bar{K}}{(n-2)\sqrt{\left(\frac{x}{n-2}+4\bar{K}\right)^2+(4n-16)\bar{K}^2}}-\frac{2x}{n}\right)\\
	&=\frac{2n\bar{K}}{n-2}\left(\frac{x^2+6x(n-2)\bar{K}+2n(n-2)^2\bar{K}^2-\frac{x}{n}(n-2)\sqrt{x^2+8x(n-2)\bar{K}+4n(n-2)^2\bar{K}^2}}{\sqrt{x^2+8x(n-2)\bar{K}+4n(n-2)^2\bar{K}^2}}\right)\\
	&=\frac{2n\bar{K}}{n-2}\left(\frac{x^2+6x(n-2)\bar{K}+2n(n-2)^2\bar{K}^2}{\sqrt{x^2+8x(n-2)\bar{K}+4n(n-2)^2\bar{K}^2}}-\frac{x(n-2)}{n}\right).
	\end{align*}
Combining the two equations above, we obtain
	\begin{align}\label{1}
	&n\bar{K}(\mathring{a}+x\mathring{a}^\prime)-a(\mathring{a}-x\mathring{a}^\prime)\nonumber\\
	&=\frac{2n\bar{K}}{n-2}\left(\frac{x^2+6x(n-2)\bar{K}+2n(n-2)^2\bar{K}^2}{\sqrt{x^2+8x(n-2)\bar{K}+4n(n-2)^2\bar{K}^2}}-\frac{x(n-2)}{n}-\frac{2x}{n}-2(n-2)\bar{K}\right)\nonumber\\
	&=\frac{2n\bar{K}}{n-2}\left(\frac{x^2+6x(n-2)\bar{K}+2n(n-2)^2\bar{K}^2}{\sqrt{x^2+8x(n-2)\bar{K}+4n(n-2)^2\bar{K}^2}}-x-2(n-2)\bar{K}\right)\nonumber\\
	&=\frac{2n\bar{K}}{(n-2)^2}\left(\frac{x^2+6x(n-2)\bar{K}+2n(n-2)^2\bar{K}^2)a-(x+2(n-2)\bar{K})a^2(n-2)}{a^2}\right)\nonumber\\
	&=\frac{2n\bar{K}}{(n-2)^2}\left(\frac{x^2+6x(n-2)\bar{K}+2n(n-2)^2\bar{K}^2)^2-(x+2(n-2)\bar{K})^2)a^2(n-2)^2}{(x^2+6x(n-2)\bar{K}+2n(n-2)^2\bar{K}^2)a+(x+2(n-2)\bar{K})a^2(n-2)}\right).
	\end{align}
Working solely on the numerator and using \eqref{a(x)}, we see that
	\begin{align}\label{2}
	&\frac{2n\bar{K}}{(n-2)^2}\left( x^2+6x(n-2)\bar{K}+2n(n-2)^2\bar{K}^2)^2-(x+2(n-2)\bar{K})^2)a^2(n-2)^2 \right)\nonumber\\
	&=\frac{2n\bar{K}}{(n-2)^2}\Big( x^2+6x(n-2)\bar{K}+2n(n-2)^2\bar{K}^2)^2\nonumber\\
	&-(x+2(n-2)\bar{K})^2)(x^2+8x\bar{K}(n-2)+4n(n-2)^2\bar{K}^2)\Big)\nonumber\\
	&=\frac{2n\bar{K}}{(n-2)^2}\left((8n-32)x(n-2)^3\bar{K}^3+n(4n-16)(n-2)^4\bar{K}^4  \right)\nonumber\\
	&=8n(n-4)(n-2)(2x+n(n-2)\bar{K})\bar{K}^4.
	\end{align}
From \eqref{1} and \eqref{2}, we have that
	\begin{align*}
	&\frac{2n\bar{K}}{(n-2)^2}\left(\frac{x^2+6x(n-2)\bar{K}+2n(n-2)^2\bar{K}^2)^2-(x+2(n-2)\bar{K})^2)a^2(n-2)^2}{(x^2+6x(n-2)\bar{K}+2n(n-2)^2\bar{K}^2)a+(x+2(n-2)\bar{K})a^2(n-2)}\right)\\
	&=\frac{8n(n-4)(n-2)(2x+n(n-2)\bar{K})\bar{K}^4}{(x^2+6x(n-2)\bar{K}+2n(n-2)^2\bar{K}^2)a+(x+2(n-2)\bar{K})a^2(n-2)}\\
	&\ge\frac{8n(n-4)(n-2)^2\bar{K}^4}{(x+\sqrt{8+2\sqrt{4n}}(n-2)\bar{K})^2}.
	\end{align*}
The last inequality above is equivalent to
	\begin{align*}
	&x^3+\left(4\sqrt{8+2\sqrt{4n}}+n-10\right)x^2(n-2)\bar{K}+2\left(2\sqrt{4n}+n\sqrt{8+2\sqrt{4n}}-2n-8\right)x(n-2)^2\bar{K}^2\\
	&+\left(2\sqrt{4n}-8\right)n(n-2)^3\bar{K}^3+16x+8n(n-2)\bar{K}\\
	&\ge(n-2)a(x^2+6x(n-2)\bar{K}+2n(n-2)^2\bar{K}^2).
	\end{align*}
iv) By setting $A=(n-2)a=\sqrt{x^2+8x(n-2)\bar{K}+4n(n-2)^2\bar{K}^2}$, we compute
	\begin{align}\label{3}
	&3x+2\sqrt{4n}(n-2)\bar{K}-\frac{3x^2+20(n-2)x\bar{K}+8n(n-2)^2\bar{K}^2}{\sqrt{x^2+8x(n-2)\bar{K}+4n(n-2)^2\bar{K}^2}}\nonumber\\
	&=\frac{(3x+2\sqrt{4n}(n-2)\bar{K})A-(3x^2+20(n-2)x\bar{K}+8n(n-2)^2\bar{K}^2)}{A}\nonumber\\
	&=\frac{(3x+2\sqrt{4n}(n-2)\bar{K})^2A^2-(3x^2+20(n-2)x\bar{K}+8n(n-2)^2\bar{K}^2)^2}{(3x+2\sqrt{4n}(n-2)\bar{K})A^2+(3x+20(n-2)x\bar{K}+8n(n-2)^2\bar{K}^2)A}.
	\end{align}
Focusing solely on the numerator, we see that
	\begin{align}\label{4}
	&(3x+2\sqrt{4n}(n-2)\bar{K})^2A^2-(3x^2+20(n-2)x\bar{K}+8n(n-2)^2\bar{K}^2)^2=12x^3(\sqrt{4n}-4)(n-2)\bar{K}\nonumber\\
	&+8x(n-2)\bar{K}\left(\frac{n}{2}-50+12\sqrt{4n}\right)x(n-2)\bar{K}+4n(n-2)^2\bar{K}^2(12\sqrt{4n}-48)x(n-2)\bar{K}\nonumber\\
	&\ge\left(12(\sqrt{4n}-4)(n-2)x\bar{K}\right)A^2.
	\end{align}
From \eqref{3} and \eqref{4}, we see that
	\begin{align}\label{5}
	&3x+2\sqrt{4n}(n-2)\bar{K}-\frac{3x^2+20(n-2)x\bar{K}+8n(n-2)^2\bar{K}^2}{\sqrt{x^2+8x(n-2)\bar{K}+4n(n-2)^2\bar{K}^2}}\nonumber\\
	&\ge\frac{\left(12(\sqrt{4n}-4)(n-2)x\bar{K}\right)A}{(3x+2\sqrt{4n}(n-2)\bar{K})A+(3x+20(n-2)x\bar{K}+8n(n-2)^2\bar{K}^2)}\ge\frac{6(\sqrt{4n}-4)(n-2)x\bar{K}}{3x+2\sqrt{4n}(n-2)\bar{K}}.
	\end{align}
Then,
	\begin{align*}
	2\mathring{a}-\frac{x}{n}+x\mathring{a}^\prime&=\frac{3x^2+20(n-2)x\bar{K}+8n(n-2)^2\bar{K}^2}{(n-2)A}-\frac{4x}{n}\\
	&\le2\sqrt{4n}\bar{K}-\frac{(n-8)x}{n(n-2)}-\frac{6(\sqrt{4n}-4)x\bar{K}}{3x+2\sqrt{4n}(n-2)\bar{K}},
	\end{align*}
where the last inequality is equivalent to the last inequality in \eqref{5}.

v) We compute
	\begin{align*}
	&\frac{x}{n-2}(a+n \bar{K})-\left(\frac{x}{n-2}+4 n \sqrt{n}\bar{K}\right)\left(\mathring{a}+a-n \bar{K}-x \mathring{a}^{\prime}\right)\\
	&=\frac{x}{n-2}\left(2n\bar{K}+x\mathring{a}'-\mathring{a}\right)+4n\sqrt{n}\bar{K}\left(n\bar{K}+x\mathring{a}'-\mathring{a}-a\right).
	\end{align*}
Using \eqref{ineqa} and \eqref{a'(x)}, we have
	\begin{align*}
	&\frac{x}{n-2}\left(2n\bar{K}+x\mathring{a}'-\mathring{a}\right)+4n\sqrt{n}\bar{K}\left(n\bar{K}+x\mathring{a}'-\mathring{a}-a\right)\\
	&<\frac{x}{n-2}\left( 2n\bar{K}+\frac{2x}{n(n-2)}-\left(\frac{2x}{n(n-2)}+4\bar{K}\right)\right)\\
	&+4n\sqrt{n}\left(n\bar{K}+\frac{2x}{n(n-2)}-\left(\frac{2x}{n(n-2)}+4\bar{K}\right)-\left(\frac{x}{n-2}+4\bar{K}\right)\right)\\
	&=\frac{x}{n-2}\left(2n\bar{K}-4\bar{K}\right)+4n\sqrt{n}\bar{K}\left(n\bar{K}-8\bar{K}-\frac{x}{n-2}\right)\\
	&=-\frac{2 x \bar{K}}{n-2}\left (2n\sqrt{n}-n+2\right)+4 n\sqrt{n} (n-8) \bar{K}^2.
	\end{align*}
\end{proof}
\begin{lemma}\label{3.1dp}
For $x\ge0$ and $n\ge8$, the following inequality holds.
	\begin{align*}
	a(\mathring{a}-x\mathring{a}^\prime)-n\bar{K}(\mathring{a}+x\mathring{a}^\prime)+P_2\left(2\mathring{a}-\frac{x}{n}+x\mathring{a}^\prime\right)-\frac{3}{2}P_2^2<0.
	\end{align*}
\end{lemma}
\begin{proof}
We assume $\bar{K}=1$, without loss of generality. The discriminant of the binomial satisfies
	\begin{align*}
	&\Delta=\left(2\mathring{a}-\frac{x}{n}+x\mathring{a}^\prime\right)-\sqrt{6(n(\mathring{a}+x\mathring{a}^\prime)-a(\mathring{a}-x\mathring{a}^\prime))}\\
	&=2\sqrt{4n}-\frac{(n-8)x}{n(n-2)}-\frac{6(\sqrt{4n}-4)x}{3x+2\sqrt{2n}(n-2)}-\frac{4\sqrt{n}(n-2)\sqrt{3(n-4)}}{x+\sqrt{8+2\sqrt{4n}}(n-2)},
	\end{align*}
using (iii) and (iv) from Lemma \ref{inequalities}. We will prove that the above discriminant is negative. We compute
	\begin{align}\label{discriminantinsidelemma}
	&\left(x+\sqrt{8+2\sqrt{4n}}(n-2)\right)\left(2\sqrt{4n}-\frac{(n-8)x}{n(n-2)}-\frac{6(\sqrt{4n}-4)x}{3x+2\sqrt{2n}(n-2)}\right)\nonumber\\
	&=\left(x+\sqrt{8+2\sqrt{4n}}(n-2)\right)\left(2\sqrt{4n}-\frac{(n-8)x}{n(n-2)}\right)\nonumber\\
	&-\left(x+\sqrt{8+2\sqrt{4n}}(n-2)\right)\frac{6(\sqrt{4n}-4)x}{3x+2\sqrt{2n}(n-2)}\nonumber\\
	&<\left(x+\sqrt{8+2\sqrt{4n}}(n-2)\right)\left(2\sqrt{4n}-\frac{(n-8)x}{n(n-2)}\right)-6(\sqrt{4n}-4)x\frac{\sqrt{8+2\sqrt{4n}}}{2\sqrt{4n}}.
	\end{align}
We need to show that for the last line of \eqref{discriminantinsidelemma}, we have
	\begin{align}\label{eqref3.2}
	&\left(x+\sqrt{8+2\sqrt{4n}}(n-2)\right)\left(2\sqrt{4n}-\frac{(n-8)x}{n(n-2)}\right)-6(\sqrt{4n}-4)x\frac{\sqrt{8+2\sqrt{4n}}}{2\sqrt{4n}}\nonumber\\
	&<4\sqrt{n}(n-2)\sqrt{3(n-4)}.
	\end{align}
The inequality \eqref{eqref3.2} is equivalent to
	\begin{align}\label{binomial}
	&-\frac{n-8}{n(n-2)}x^2+\left(2\sqrt{4n}-\left(4-\frac{8}{n}-3\sqrt{\frac{4}{n}}\right)\sqrt{8+2\sqrt{4n}}\right)x\nonumber\\
	&<2\sqrt{n}(n-2)\left(\sqrt{12(n-4)}-\sqrt{32+8\sqrt{4n}}\right).
	\end{align}
	
For $x\ge 0$ and $n\ge 8$, we will prove that the discriminant of \eqref{binomial} in negative. It is enough to show that 	
\begin{align}\label{negativediscr}
	2\sqrt{4n}-\left(4-\frac{8}{n}-3\sqrt{\frac{4}{n}}\right)\sqrt{8+2\sqrt{4n}}-\sqrt{\frac{8(n-8)}{\sqrt{n}}\left(\sqrt{12(n-4)}-\sqrt{32+8\sqrt{4n}}\right)}<0.
	\end{align}
We set $g(n)=\frac{\left(4-\frac{8}{n}-3\sqrt{\frac{4}{n}}\right)}{2}=2-\frac{4}{n}-\frac{3}{\sqrt{n}}$ and $h(n)=\frac{\sqrt{8+2\sqrt{4n}}}{\sqrt{4n}}=\sqrt{\frac{2}{n}+\frac{1}{\sqrt{n}}}$. From the monotonicity of $g(n)$ and $h(n)$, for $8\le n\le100$, we have the following.
	\begin{align*}
	&1-g(n)h(n)=1-2\sqrt{\frac{2}{n}+\frac{1}{\sqrt{n}}}+\left(\frac{4}{n}+\frac{3}{\sqrt{n}}\right)\sqrt{\frac{2}{n}+\frac{1}{\sqrt{n}}}\\
	&\le1-2\sqrt{\frac{2}{100}+\frac{1}{\sqrt{100}}}+\left(\frac{4}{8}+\frac{3}{\sqrt{8}}\right)\sqrt{\frac{2}{8}+\frac{1}{\sqrt{8}}}\\
	&=1.52=\sqrt{2.31}.
	\end{align*}
On the other hand, for $8\le n\le 100$, we have
	\begin{align*}
&\frac{\frac{8(n-8)}{\sqrt{n}}\left(\sqrt{12(n-4)}-\sqrt{32+8\sqrt{4n}}\right)}{\left(2\sqrt{4n}\right)^2}=\left(\frac{1}{2}-\frac{4}{n}\right)\left(\sqrt{12\left(1-\frac{4}{n}\right)}-\sqrt{\frac{32}{n}+\frac{16}{\sqrt{n}}}\right)\nonumber\\
&\le \frac{1}{2}\left(\sqrt{12\left(1-\frac{4}{100}\right)}-\sqrt{\frac{32}{100}+\frac{16}{\sqrt{100}}}\right)-\frac{4}{100}\sqrt{12\left(1-\frac{4}{8}\right)}+\frac{4}{8}\sqrt{\frac{32}{8}+\frac{16}{\sqrt{8}}}\nonumber\\
&<2.46
	\end{align*}
and so,
	\begin{align}\label{firstnumber}
	2.31<\left(\frac{1}{2}-\frac{4}{n}\right)\left(\sqrt{12\left(1-\frac{4}{n}\right)}-\sqrt{\frac{32}{n}+\frac{16}{\sqrt{n}}}\right)<2.46.
	\end{align}
In the same way,
	\begin{align}
1.249&<\left(\frac{1}{2}-\frac{4}{n}\right)\left(\sqrt{12\left(1-\frac{4}{n}\right)}-\sqrt{\frac{32}{n}+\frac{16}{\sqrt{n}}}\right)<1.787, \text { for } n\ge 101.\label{secondnumber} 
	\end{align}
Set $k_n= \begin{cases}2.31, & 8 \leq n \leq 100, \\ 1.249, & n\ge 101. \end{cases}$ The following inequalities follow from \eqref{firstnumber} and \eqref{secondnumber}
	\begin{align*}
&1-g(n) h(n) <\sqrt{k_n}, \\
k_n & <\left(\frac{1}{2}-\frac{4}{n}\right)\left(\sqrt{12\left(1-\frac{4}{n}\right)}-\sqrt{\frac{32}{n}+\frac{16}{\sqrt{n}}}\right),
	\end{align*}
which means that \eqref{negativediscr} is satisfied and this completes the proof.
\end{proof}
We denote $\mathring{a}^{\prime}\left(|H|^2\right)$ and $\mathring{a}^{\prime \prime}\left(|H|^2\right)$ by $\mathring{a}^{\prime}$ and $\mathring{a}^{\prime \prime}$, respectively. Then the evolution equation of $\mathring{a}$ satisfies
	\begin{align}\label{evolutionofmathringa}
\frac{\partial}{\partial t} \mathring{a}=\Delta \mathring{a}+2 \mathring{a}^{\prime} \cdot\left(-|\nabla H|^2+R_2+n \bar{K}|H|^2\right)-\mathring{a}^{\prime \prime} \cdot|\nabla| H|^2|^2 .
	\end{align}
For $0<\varepsilon \ll \frac{1}{n^{5 / 2}}$, set
	\begin{align}\label{omega}
\omega=\frac{|H|^2}{n-2}+4 n \sqrt{n} \bar{K},
	\end{align}
where from \eqref{ineqa} and \eqref{a'(x)}, we have
	\begin{align*}
\omega>n \mathring{a}>a.
	\end{align*}
In the following theorem, we prove the preservation of the quartic pinching condition under the mean curvature flow.
\begin{proposition}[cf.\cite{DongPu}, Proposition 3.2]\label{preservation}
Let $F_0: \mathcal{M} \rightarrow \mathbb{S}^{n+m}\left(\frac{1}{\sqrt{\bar{K}}}\right)$ be an $n$-dimensional, compact submanifold immersed in the sphere, $n\ge 8,m\ge2$. Suppose there exists $0<\varepsilon\ll \frac{1}{n^{5 / 2}}$, such that 
	\begin{align}\label{pinching}
	|A|^2 \leq a\left(|H|^2\right)-\varepsilon \omega.
	\end{align}
Then, \eqref{pinching} holds along the mean curvature flow for every $t \in[0, T)$, where $T \leq \infty$.
\end{proposition}
\begin{proof}
Let $f=|\mathring{A}|^2-\mathring{a}+\varepsilon \omega$. By \eqref{eqn_H}, \eqref{eqn_|A|^2traceless} and \eqref{evolutionofmathringa}, we have
	\begin{align*}
& \left(\partial_t-\Delta\right) f=-2|\nabla \mathring{A}|^2+2\left(\mathring{a}^{\prime}-\frac{\varepsilon}{n-2}\right)|\nabla H|^2+\mathring{a}^{\prime \prime}|\nabla| H|^2|^2+2 R_1-\frac{2}{n} R_2-2 n \bar{K} |\mathring{A}|^2\\
&-2\left(\mathring{a}^{\prime}-\frac{\varepsilon}{n-2}\right) \left(R_2+n \bar{K}|H|^2\right)\\
& \leq 2\left(-\frac{2(n-1)}{n(n+2)}+\mathring{a}^{\prime}-\frac{\varepsilon}{n-2}+2|H|^2 \mathring{a}^{\prime \prime}\right)|\nabla H|^2+2|\mathring{A}|^2\left(|\mathring{A}|^2+\frac{1}{n}|H|^2-n \bar{K}\right)\\
&+2 P_2\left(2|\mathring{A}|^2-\frac{1}{n}|H|^2-\frac{3}{2} P_2\right)-2\left(\mathring{a}^{\prime}-\frac{\varepsilon}{n-2}\right)\cdot|H|^2\left(|\mathring{A}|^2+\frac{1}{n}|H|^2+n \bar{K}-P_2\right) .
	\end{align*}
From Lemma \ref{inequalities} (ii), the coefficient of $|\nabla H|^2$ is negative. Replacing $|\mathring{A}|^2$ with $f+\mathring{a}-\varepsilon\omega$, the above formula becomes
	\begin{align*}
& \left(\partial_t-\Delta\right) f \leq 2 f\left(2 \mathring{a}+\frac{1}{n}|H|^2-n \bar{K}-\mathring{a}^{\prime}|H|^2+2 P_2+\varepsilon\left(\frac{|H|^2}{n-2}-2 \omega\right)\right)+2 f^2\\
&+2\left(a\left(\mathring{a}-|H|^2 \mathring{a}^{\prime}\right)-n \bar{K}\left(\mathring{a}+|H|^2 \mathring{a}^{\prime}\right)\right) \\
& +2 P_2\left(2 \mathring{a}-\frac{1}{n}|H|^2+|H|^2 \mathring{a}^{\prime}-\varepsilon\left(\frac{|H|^2}{n-2}+2 \omega\right)-\frac{3}{2} P_2\right) \\
& +2 \varepsilon\left(\frac{|H|^2}{n-2} \cdot(a+n \bar{K})-\omega\left(a+\mathring{a}-n \bar{K}-\mathring{a}^{\prime}|H|^2\right)\right) +2 \varepsilon^2 \omega\left(\omega-\frac{|H|^2}{n-2}\right) .
	\end{align*}
From Lemma \ref{inequalities} (v) and the choice of $\varepsilon$, we have
	\begin{align*}
& \varepsilon\left(\frac{|H|^2}{n-2} \cdot(a+n \bar{K})-\omega\left(a+\mathring{a}-n \bar{K}-\mathring{a}^{\prime}|H|^2\right)\right)+\varepsilon^2 \omega\left(\omega-\frac{|H|^2}{n-2}\right) \\
& =\varepsilon\left(\frac{|H|^2}{n-2} \cdot(a+n \bar{K})-\omega\left(a+\mathring{a}-n \bar{K}-\mathring{a}^{\prime}|H|^2\right)+4 n \sqrt{n} \bar{K} \varepsilon \omega\right) \\
& <2 \varepsilon\left(2n \sqrt{n} \bar{K}^2(n-8+4 \varepsilon n \sqrt{n})-\frac{|H|^2 \bar{K}}{n-2}(2n \sqrt{n}-n+2-2\varepsilon n \sqrt{n})\right)\\
& <2 \varepsilon\left(2n \sqrt{n} \bar{K}^2(n-4)-\frac{|H|^2 \bar{K}}{n-2}(2n \sqrt{n}-n)\right),
	\end{align*}
where the last inequality comes from the fact that $\varepsilon\le\frac{1}{\sqrt[5]{n}}$. So, using Lemma \ref{inequalities} (iii), we have
	\begin{align}\label{3.10dp}
&2 \varepsilon\left(2n \sqrt{n} \bar{K}^2(n-4)-\frac{|H|^2 \bar{K}}{n-2}(2n \sqrt{n}-n)\right) =64\varepsilon n^2 \sqrt{n} \bar{K}\left(\frac{(n-4) \bar{K}}{16n}-\frac{|H|^2}{16n}\frac{2\sqrt{n}-1}{2\sqrt{n}(n-2)}\right) \nonumber\\
& <64 \varepsilon n^2 \sqrt{n} \bar{K}\left(\frac{n(n-4)(n-2)^2 \bar{K}^3}{\left(|H|^2+4 n(n-2) \bar{K}\right)^2}-\frac{|H|^2}{16n}\frac{2\sqrt{n}-1}{2\sqrt{n}(n-2)}\right)\nonumber \\
&<64 \varepsilon n^2 \sqrt{n} \bar{K}\left(\frac{n(n-4)(n-2)^2 \bar{K}^3}{\left(|H|^2+4 n(n-2) \bar{K}\right)^2}\right) \nonumber\\
& <8 n^2 \sqrt{n} \varepsilon\left(\frac{8n(n-4)(n-2)^2\bar{K}^4}{\left(|H|^2+\sqrt{8+2\sqrt{4n}}(n-2)\bar{K}\right)^2}\right).\end{align}
From \eqref{3.10dp}, Lemmas \ref{inequalities} (iii), \ref{3.1dp} and the choice of $\varepsilon$, we have
	\begin{align*}
& a\left(\mathring{a}-|H|^2 \mathring{a}^{\prime}\right)-n \bar{K}\left(\mathring{a}+|H|^2 \mathring{a}^{\prime}\right) +P_2\left(2 \mathring{a}-\frac{1}{n}|H|^2+|H|^2 \mathring{a}^{\prime}-\varepsilon\left(\frac{|H|^2}{n-2}+2 \omega\right)-\frac{3}{2} P_2\right) \\
& +\varepsilon\left(\frac{|H|^2}{n-2}(a+n \bar{K})-\omega\left(a+\mathring{a}-n \bar{K}-\mathring{a}^{\prime}|H|^2\right)\right)+\varepsilon^2 \omega\left(\omega-\frac{|H|^2}{n-2}\right) \\
& <a\left(\mathring{a}-|H|^2 \mathring{a}^{\prime}\right)-n \bar{K}\left(\mathring{a}+|H|^2 \mathring{a}^{\prime}\right)+P_2\left(2 \mathring{a}-\frac{1}{n}|H|^2+|H|^2 \mathring{a}^{\prime}\right)-\frac{3}{2} P_2^2 \\
&+8 n^2 \sqrt{n} \varepsilon\left(\frac{8n(n-4)(n-2)^2\bar{K}^4}{\left(|H|^2+\sqrt{8+2\sqrt{4n}}(n-2)\bar{K}\right)^2}\right)\\
&<0,
	\end{align*}
for $\varepsilon\to0$. The proposition follows from the maximum principle.
\end{proof}
\section{Gradient estimates}\label{gradientestimates}
This section presents a proof of the gradient estimate for the mean curvature flow. We establish this estimate directly from the quartic curvature bound $|A|^2 < a-\varepsilon\omega$, without relying on the asymptotic cylindrical estimates. In fact, we demonstrate that the cylindrical estimates follow as a consequence of the gradient estimates we derive here. 
\\
The gradient estimates are pointwise estimates that rely solely on the mean curvature (or, equivalently, the second fundamental form) at a point and not on the maximum of curvature, as is the case with more general parabolic-type derivative estimates. The importance of these estimates is that they allow us to control the mean curvature and, hence, the full second fundamental form on a neighbourhood of fixed size.
\begin{theorem}[cf.\cite{HuSi09}, Section 6]\label{thm_gradient}
Let $ \mathcal{M}_t , t \in [0,T)$ be a closed $n$-dimensional, quartically bounded solution to the mean curvature flow in $\mathbb{S}^{n+m}$, that is
	\begin{align*}
	|A|^2 -a+\varepsilon\omega<0.
	\end{align*}
 Then, there exists a constant $ \gamma_1= \gamma_1 (n, \mathcal M_0)$ and a constant $ \gamma_2 = \gamma_2 ( n , \mathcal{M}_0)$, such that the flow satisfies the uniform estimate
	\begin{align*}
	|\nabla A|^2 \leq \gamma_1 |A|^4+\gamma_2,
	\end{align*}
 for every $t\in [0, T)$.
\end{theorem}
\begin{proof}
We consider here the evolution equation for $\frac{|\nabla A|^2}{g^2},$ where $ g = a-|A|^2-\varepsilon\omega>0$. Since $ |A|^2-a < 0$ and $\mathcal{M}_0$ is compact, there exists an $ \eta(\mathcal{M}_0) >0, C_\eta(\mathcal{M}_0)>0$, so that
	\begin{align*}
	a(1-\eta)-|A|^2 \geq C_\eta>0.
	\end{align*}
Hence, we set $g=a-|A|^2\ge \eta a>\eta|A|^2>\e_1|A|^2+\e_2,$ where $ \e_1 < \eta$ and $\e_2>0$. From Proposition \ref{preservation}, we get
\begin{align*}
 \left(\partial_t-\Delta\right) g&=2\left(|\nabla A|^2-\left(\mathring{a}^\prime-\frac{\varepsilon}{n-2}+2|H|^2\mathring{a}^{\prime\prime}\right)|\nabla H|^2\right) -2 R_1+\frac{2}{n} R_2+2 n \bar{K} |A|^2\\
&+2\left(\mathring{a}^{\prime}-\frac{\varepsilon}{n-2}\right) \left(R_2+n \bar{K}|H|^2\right)\\
&\ge 2\left(|\nabla A|^2-\left(\mathring{a}^\prime-\frac{\varepsilon}{n-2}+2|H|^2\mathring{a}^{\prime\prime}\right)|\nabla H|^2\right).
	\end{align*}
Using Lemma \ref{katoinequality} and Lemma \ref{inequalities} (ii), we arrive at
	\begin{align*}
 \left(\partial_t-\Delta\right) g&\ge 2\left(|\nabla A|^2-\left(\frac{2(n-1)}{n(n+2)}-\frac{\varepsilon}{n-2}\right)|\nabla H|^2\right)\\
 &\ge 2\left(1-\frac{n+2}{3}\left(\frac{2(n-1)}{n(n+2)}-\frac{\varepsilon}{n-2}\right)\right)|\nabla A|^2\\
 &\ge\frac{2(n+2)}{3n}|\nabla A|^2,
	\end{align*}
where it remains positive for any $n$.
The evolution equation for $ |\nabla A|^2 $ is given by  	\begin{align*}
	\left(\partial_t-\Delta\right) |\nabla A|^2&=-2 |\nabla^2 A|^2+c |A|^2 |\nabla A|^2+d|\nabla A|^2+c_1|A||\nabla A|+c_2|\nabla A|.
	\end{align*}
If $w,z$ satisfy the evolution equations $\partial_tw = \Delta w+W$ and $\partial_tz = \Delta z+Z$, then we find
 	\begin{align*}
	\left(\partial_t-\Delta\right)\frac{w}{z}&=\frac{2}{z}\left\la \nabla \left( \frac{w}{z}\right) , \nabla z \right\ra+\frac{W}{z}-\frac{w}{z^2} Z=2\frac{\la \nabla w , \nabla z \ra}{z^2}-2 \frac{w|\nabla z |^2}{z^3}+\frac{W}{z}-\frac{w}{z^2} Z.
	\end{align*}
Furthermore, for any function $g$, we have by Kato's inequality
	\begin{align*}
	\la \nabla g , \nabla |\nabla A|^2 \ra&\leq 2 |\nabla g| |\nabla^2 A| |\nabla A| \leq \frac{1}{g}|\nabla g |^2 | \nabla A|^2+g |\nabla^2 A|^2.
	\end{align*}
We then get
	\begin{align*}
	-\frac{2}{g}| \nabla^2 A|^2+\frac{2}{g}\left\la \nabla g ,\nabla \left( \frac{|\nabla A|^2}{g}\right) \right\ra \leq-\frac{2}{g}| \nabla^2 A|^2-\frac{2}{g^3}|\nabla g|^2 |\nabla A|^2+\frac{2}{g^2}\la \nabla g ,\nabla |\nabla A|^2 \ra \leq 0.
	\end{align*}
Then, if we let $ w = |\nabla A|^2 $ and $ z = g$, with $W=-2 |\nabla^2 A|^2+c |A|^2 |\nabla A|^2+d|\nabla A|^2+c_1|A||\nabla A|+c_2|\nabla A|$ and $Z\geq \frac{2(n+2)}{3n}| \nabla A|^2 $, we get
	\begin{align*}
	\left(\partial_t-\Delta\right) \frac{|\nabla A|^2}{g}&\leq \frac{2}{g}\left\la \nabla g ,\nabla \left( \frac{|\nabla A|^2}{g}\right) \right\ra+\frac{1}{g}\Big(-2 |\nabla^2 A|^2+c |A|^2 |\nabla A|^2 +d|\nabla A|^2\\
	&+c_1|A||\nabla A|+c_2|\nabla A|\Big)-\frac{2(n+2)}{3n}\frac{|\nabla A|^4}{g^2} \\
	&\leq c |A|^2 \frac{|\nabla A|^2}{g}+d\frac{|\nabla A|^2}{g}-\frac{2(n+2)}{3n}\frac{|\nabla A|^4}{g^2}.
	\end{align*}
We repeat the above computation with $w = \frac{|\nabla A|^2}{g}, z = g,$
	\begin{align*}
	W\leq c |A|^2 \frac{|\nabla A|^2}{g}+d\frac{|\nabla A|^2}{g}-\frac{2(n+2)}{3n}\frac{|\nabla A|^4}{g^2}+c_1|A|\frac{|\nabla A|}{g}+c_2\frac{|\nabla A|}{g}\end{align*}
and $ Z\geq 0$ and assuming $\frac{|\nabla A|}{g}>1$, we get
	\begin{align*}
	\left(\partial_t-\Delta\right)\frac{|\nabla A|^2}{g^2}&\leq \frac{2}{g}\left\la \nabla g ,\nabla \left( \frac{|\nabla A|^2}{g^2}\right) \right\ra+\frac{1}{g}\left( c |A|^2 \frac{|\nabla A|^2}{g}+d\frac{|\nabla A|^2}{g}-\frac{2(n+2)}{3n}\frac{|\nabla A|^4}{g^2}\right)\\
	&+c_3|A|^2\frac{|\nabla A|^2}{g^2}+c_4\frac{|\nabla A|^2}{g^2}.
 	\end{align*}
The nonlinearity then is
	\begin{align*}
	\frac{|\nabla A|^2}{g^2} \left( (c+c_3)|A|^2+d+c_4-\frac{2(n+2)}{3n}\frac{|\nabla A|^2}{g} \right).
	\end{align*}
Since $g>\e_1|A|^2+\e_2$, there exists a constant $N$, such that
	\begin{align*}
	Ng\ge (c+c_3)|A|^2+d+c_4.
	\end{align*}
Hence, by the maximum principle, there exists a constant (with $\eta,\e_1,\e_2$ chosen sufficiently small so that N is sufficiently large, this estimate holds at the initial time), such that
	\begin{align*}
	\frac{|\nabla A|^2}{g^2}\leq \frac{3nN}{2(n+2)}.
	\end{align*}
Therefore, we see there exists a constant $\mathcal{C} = \frac{3nN}{2(n+2)}= \mathcal{C}(n, \mathcal{M}_0) $, such that
	\begin{align*}
	\frac{|\nabla A|^2}{g^2}\leq \mathcal{C}
	\end{align*}
and from the definition of $g$, we get the result of the lemma.
\end{proof}
\begin{theorem} [\cite{HNAV}, Theorem 4.2]\label{thm_gradient2} Let $ \mathcal{M}_t , t \in [0,T)$ be a closed $n$-dimensional, quartically bounded solution to the mean curvature flow in $\mathbb{S}^{n+m}$, that is
	\begin{align*}
	|A|^2 -a+\varepsilon\omega<0.
	\end{align*}
 Then, there exist constants $\gamma_3, \gamma_4$ depending only on the dimension $n$ and $\mathcal{M}_0$, such that the flow satisfies the uniform estimate
 	\begin{align*}
	|\nabla^2 A|^2 \leq \gamma_3|A|^6+\gamma_4,
	\end{align*}
for any $t\in[0,T)$.
\end{theorem}
Higher order estimates on $\left|\nabla^m A\right|$ for all $m$ follow by an analogous method. Furthermore, we derive estimates on the time derivative of the second fundamental form, since
	\begin{align*}
	\left|\partial_t A\right|=|\Delta A+A * A * A| \leq C|\nabla^2 A|+C|A|^3 \leq c_1|A|^3+c_2.
	\end{align*}
\section{Convergence for finite time}
In this section, we use a blow up argument to prove a codimension and a cylindrical estimate. In particular, we show that in regions of high curvature, the submanifold becomes approximately codimension one, quantitatively, and is weakly convex and moves by translation or is a self shrinker. 
\\
We need the evolution equation of the quantity $\frac{|A^-|^2}{f}$. In this section we assume $f:=-|A|^2+a-\varepsilon\omega$, where $a,\omega$ come from \eqref{a(x)} and \eqref{omega}. Firstly, we separate the second fundamental form in the principal direction and the second fundamental form in the orthogonal direction and compute their evolution equations separately. Later, we find estimates for the reaction and gradient terms, as well as for the lower order terms. Since in the limit the background space is Euclidean, the result will follow from the maximum principle.
\\
Lastly, we give a full classification for singularity models of quartically pinched solutions in high codimension mean curvature flow.
\\
Note that throughout this section we work at points where $|H|>0$.
\subsection{The evolution of $\frac{|A^-|^2}{f}$}
Here, we compute the evolution equation of the quantity $\frac{|A^-|^2}{f}$. 
 In order to prove the codimension estimate, we need good estimates for the reaction terms in this equation. These are proven following Andrews-Baker in \cite{AnBa10} and Langford-Lynch-Nguyen in \cite{LaLyNg}. From \eqref{eqn_A} and \eqref{eqn_H}, we have that the projection $\langle A, H\rangle $ satisfies
	\begin{align*}
	(\partial_t - \Delta) A_{ij}^{\alpha} H^{\alpha}&= -2\sum_{p,\alpha}\nabla_p A_{ij}^{\alpha} \nabla_p H^{\alpha} +2|H| |h|^2 h_{ij} + 2\sum_{p,q,\alpha\ge 2} |H| h_{pq} ( A_{ij}^{\alpha} A_{pq}^{\alpha} - A_{ip}^{\alpha} A_{jq}^{\alpha})\\
	&+ \sum_{p,q,\alpha\ge 2}|H| A_{iq}^{\alpha} A_{qp}^{\alpha} h_{pj} +\sum_{p,q,\alpha\ge 2} |H| A_{jq}^{\alpha} A_{qp}^{\alpha} h_{pi} + 2 \bar K |H|^2 g_{ij}.
	\end{align*}
The evolution of $|H|$ satisfies
	\begin{align*}
	(\partial_t - \Delta) |H|&= \frac{1}{|H|^3} \langle H, \nabla_i H\rangle \langle H, \nabla_i H\rangle - \frac{|\nabla H|^2}{|H|} + \frac{|\langle A, H\rangle |^2 }{|H|} + n \bar K |H|.
	\end{align*}
Inserting the identities $\frac{|\langle A, H\rangle |^2 }{|H|} = |h|^2 |H|$
and
	\begin{align*}
- \frac{|\nabla H|^2}{|H|}+ \frac{1}{|H|^3}\langle H, \nabla_i H\rangle \langle H, \nabla_i H\rangle&= - \frac{|\nabla |H||^2}{|H|} - |H| |\nabla \nu_1|^2 + \frac{1}{|H|} \langle \nu_1, \nabla_i H\rangle \langle \nu_1, \nabla_i H\rangle\\
&= - |H| |\nabla \nu_1|^2,
	\end{align*}
we obtain
	\begin{align*}
(\partial_t - \Delta) |H|&= |h|^2 |H| + n \bar K |H| - |H||\nabla \nu_1|^2.
	\end{align*}
For a tensor $B_{ij}$ divided by a positive scalar function $f$, there holds
	\begin{align*}
	(\partial_t - \Delta) \frac{B_{ij}}{f} = \frac{1}{f} (\nabla_t - \Delta) B_{ij} - \frac{B_{ij}}{f^2} (\partial_t - \Delta) f + \frac{2}{f} \bigg\langle \nabla \frac{B_{ij}}{f}, \nabla f \bigg\rangle.
	\end{align*}
Therefore, dividing $\langle A_{ij}, H\rangle$ by $|H|$, we obtain
	\begin{align*}
	(\partial_t - \Delta) h_{ij}&= |h|^2 h_{ij} + 2\sum_{p,q,\alpha\ge 2} h_{pq} ( A_{ij}^{\alpha} A_{pq}^{\alpha} - A_{ip}^{\alpha} A_{jq}^{\alpha}) +\sum_{p,q,\alpha\ge 2}A_{iq}^{\alpha} A_{qp}^{\alpha} h_{pj} +2 |H|^{-1} \langle \nabla h_{ij}, \nabla |H| \rangle\\
	&+ \sum_{p,q,\alpha\ge 2}A_{jq}^{\alpha} A_{qp}^{\alpha} h_{pi}+ 2 \bar K |H| g_{ij} - n \bar K h_{ij} -2|H|^{-1} \langle \nabla A_{ij} ,\nabla H \rangle + h_{ij} |\nabla \nu_1|^2 .
	\end{align*}
We simplify the gradient terms by decomposing
	\begin{align*}
	- 2\langle \nabla A_{ij} , \nabla H\rangle&= - 2\langle \nabla h_{ij} \nu_1 + h_{ij} \nabla \nu_1 + \nabla A_{ij}^- , \nabla |H| \nu_1 + 2|H| \nabla \nu_1 \rangle\\
	&= - 2\langle \nabla h_{ij} , \nabla |H|\rangle - 2|H| h_{ij} |\nabla \nu_1|^2 - 2\langle \nabla A_{ij}^- , \nabla |H| \nu_1 \rangle- 2|H| \langle \nabla A_{ij}^- , \nabla \nu_1 \rangle
	\end{align*}
and so obtain
	\begin{align*}
	(\partial_t - \Delta) h_{ij}&= |h|^2 h_{ij} + 2\sum_{p,q,\alpha\ge 2} h_{pq} ( A_{ij}^{\alpha} A_{pq}^{\alpha} - A_{ip}^{\alpha} A_{jq}^{\alpha}) +\sum_{p,q,\alpha\ge 2} A_{iq}^{\alpha} A_{qp}^{\alpha} h_{pj} - 2 \langle \nabla A_{ij}^- , \nabla \nu_1 \rangle\\
	& +\sum_{p,q,\alpha\ge 2} A_{jq}^{\alpha} A_{qp}^{\alpha} h_{pi} + 2 \bar K |H| g_{ij} - n \bar K h_{ij} - h_{ij} |\nabla \nu_1|^2 - 2|H|^{-1}\langle \nabla A_{ij}^- , \nabla |H| \nu_1 \rangle.
	\end{align*}
Next, we compute
	\begin{align*}
	(\partial_t - \Delta) |h|^2&=2\sum_{i,j} h_{ij}(\nabla_t - \Delta) h_{ij} - 2 |\nabla h|^2 \\
	&= 2|h|^4 + 4\sum_{i,j} |h_{ij}A_{ij}^- |^2 - 4\sum_{i,j,p,q,\alpha\ge 2} h_{ij} h_{pq}A_{ip}^{\alpha} A_{jq}^{\alpha} + 4\sum_{i,j,p,q,\alpha\ge 2} h_{ij} h_{pj} A_{iq}^{\alpha} A_{qp}^{\alpha}\\
	& + 4 \bar K |H|^2- 2n \bar K |h|^2- 2|\nabla h|^2 - 2 |h|^2 |\nabla \nu_1|^2 - 4\sum_{i,j}|H|^{-1} h_{ij} \langle \nabla A_{ij}^- , \nabla |H| \nu_1 \rangle \\
	&- 4\sum_{i,j}h_{ij} \langle \nabla A_{ij}^- , \nabla \nu_1 \rangle,
	\end{align*}
and following Naff in \cite{Naff}, rewrite
	\begin{align*}
	4 \sum_{i,j,p,q,\alpha\ge 2}h_{ij} h_{pj} A_{iq}^{\alpha} A_{qp}^{\alpha} - 4\sum_{i,j,p,q,\alpha\ge 2} h_{ij} h_{pq}A_{ip}^{\alpha} A_{jq}^{\alpha}&= 2\sum_{i,j,p,q}\langle h_{ij} A_{iq}^- - h_{iq}A_{ij}^- , h_{pj} A_{pq}^- - h_{pq} A_{pj}^- \rangle\\
	&= 2\sum_{i,j,p} |h_{ip} A_{pj}^- - h_{jp}A_{pi}^- |^2.
	\end{align*}
Hence,
	\begin{align*}
	(\partial_t - \Delta) |h|^2&= 2|h|^4 + 4\sum_{i,j} |h_{ij}A_{ij}^- |^2+ 2\sum_{i,j,p} |h_{ip} A_{pj}^- - h_{jp}A_{pi}^- |^2 + 4 \bar K |H|^2 - 2n \bar K |h|^2\\
	&- 2|\nabla h|^2 - 2 |h|^2 |\nabla \nu_1|^2 - 4\sum_{i,j}|H|^{-1} h_{ij} \langle \nabla A_{ij}^- , \nabla |H| \nu_1 \rangle - 4\sum_{i,j}h_{ij} \langle \nabla A_{ij}^- , \nabla \nu_1 \rangle
	\end{align*}
and since $|A^-|^2 = |A|^2 - |h|^2$,
	\begin{align*}
	(\partial_t- \Delta) |A^-|^2 &= 2 | \langle A, A\rangle |^2 -2|h|^4 - 4\sum_{i,j} |h_{ij}A_{ij}^- |^2+ 2 |R^{\perp}|^2 -2\sum_{i,j,p} |h_{ip} A_{pj}^- - h_{jp}A_{pi}^- |^2 \\
	&- 2n \bar K |A^-|^2- 2 |\nabla A|^2+ 2|\nabla h|^2 + 2 |h|^2 |\nabla \nu_1|^2 \\
	&+ 4\sum_{i,j}|H|^{-1} h_{ij} \langle \nabla A_{ij}^- , \nabla |H| \nu_1 \rangle+ 4\sum_{i,j}h_{ij} \langle \nabla A_{ij}^- , \nabla \nu_1 \rangle.
	\end{align*}
The reaction terms can be simplified by observing
\[2 | \langle A, A\rangle |^2 -2|h|^4 - 4\sum_{i,j} |h_{ij}A_{ij}^- |^2 = 2|\langle A^-, A^-\rangle |^2\]
and (recalling the decomposition of $R^\perp$ carried out above)
\[2 |R^{\perp}|^2 -2\sum_{i,j,p} |h_{ip} A_{pj}^- - h_{jp}A_{pi}^- |^2= 2 \sum_{i,j,p}|h_{ip} A_{pj}^- - h_{jp} A_{pi}^- |^2 + 2\sum_{i,j,p}|A_{ip}^- \otimes A_{pj}^- - A_{jp}^- \otimes A_{pi}^- |^2,\]
hence
	\begin{align*}
	(\partial_t - \Delta) |A^-|^2&= 2|\langle A^-, A^-\rangle |^2 + 2\sum_{i,j,p} |h_{ip} A_{pj}^- - h_{jp} A_{pi}^- |^2 + 2\sum_{i,j,p}|A_{ip}^- \otimes A_{pj}^- - A_{jp}^- \otimes A_{pi}^- |^2\\
	& - 2n \bar K |A^-|^2- 2 |\nabla A|^2+ 2|\nabla h|^2 + 2 |h|^2 |\nabla \nu_1|^2 \\
	&+ 4\sum_{i,j}|H|^{-1} h_{ij} \langle \nabla A_{ij}^- , \nabla |H| \nu_1 \rangle+ 4\sum_{i,j}h_{ij} \langle \nabla A_{ij}^- , \nabla \nu_1 \rangle.
	\end{align*}
Since $\nabla A = \nabla h \nu_1 + h \nabla \nu_1 + \nabla A^-$, we compute
	\begin{align*}
	 2 |\nabla A|^2 = 2 |\nabla h|^2 + 2 |h|^2 |\nabla \nu_1|^2 + 2 |\nabla A^-|^2 + 4\sum_{i,j} h_{ij} \langle \nabla A_{ij}^- ,\nabla \nu_1\rangle + 4\sum_{i,j} \langle \nabla A_{ij}^- , \nabla h_{ij}\nu_1\rangle 
	\end{align*}
and so obtain
	\begin{align*}
	(\partial_t - \Delta) |A^-|^2&= 2|\langle A^-, A^-\rangle |^2 + 2 \sum_{i,j,p}|h_{ip} A_{pj}^- - h_{jp} A_{pi}^- |^2 + 2\sum_{i,j,p}|A_{ip}^- \otimes A_{pj}^- - A_{jp}^- \otimes A_{pi}^- |^2\\
	& - 2n \bar K |A^-|^2- 2 |\nabla A^-|^2- 4\sum_{i,j} \langle \nabla A_{ij}^- , \nabla h_{ij}\nu_1\rangle+ 4\sum_{i,j}|H|^{-1} h_{ij} \langle \nabla A_{ij}^- , \nabla |H| \nu_1 \rangle.
	\end{align*}
Differentiating $\langle A_{ij}^- , \nu_1\rangle = 0$, we see the last two gradient terms may be expressed as
	\begin{align*}
	&- 4 \sum_{i,j}\langle \nabla A_{ij}^- ,\nabla h_{ij}\nu_1\rangle+ 4\sum_{i,j}|H|^{-1} h_{ij} \langle \nabla A_{ij}^- , \nabla |H| \nu_1 \rangle\\
	&= -\sum_{i,j,k} (4 \nabla_k h_{ij} - 4|H|^{-1} h_{ij} \nabla_k |H|)\langle \nabla_k A_{ij}^- , \nu_1 \rangle \\
	&= \sum_{i,j,k}(4 \nabla_k h_{ij} - 4|H|^{-1} h_{ij} \nabla_k |H|)\langle A_{ij}^- , \nabla_k \nu_1 \rangle
	\end{align*}
and consequently,
	\begin{align*}
	(\partial_t - \Delta) |A^-|^2&= 2\sum_{i,j,p,q}|\langle A_{ij}^- , A_{pq}^- \rangle |^2 + 2\sum_{i,j,p} |h_{ip} A_{pj}^- - h_{jp} A_{pi}^- |^2 + 2\sum_{i,j,p}|A_{ip}^- \otimes A_{pj}^- - A_{jp}^- \otimes A_{pi}^- |^2 \\
	&- 2n \bar K |A^-|^2- 2 |\nabla A^-|^2+\sum_{i,j,k} (4 \nabla_k h_{ij} - 4|H|^{-1} h_{ij} \nabla_k |H|)\langle A_{ij}^- , \nabla_k \nu_1 \rangle.
	\end{align*}
Since $0\le f=-|A|^2+a-\varepsilon \omega$, from Proposition \ref{preservation} and sending $\varepsilon\to 0$, we have
\begin{align*}
 \left(\partial_t-\Delta\right) f&=2|\nabla^\bot A|^2-2\mathring{a}^{\prime}|\nabla H|^2-\mathring{a}^{\prime \prime}|\nabla| H|^2|^2  -2 R_1+\frac{2}{n} R_2+2 n \bar{K} |A|^2+2\mathring{a}^{\prime} \left(R_2+n \bar{K}|H|^2\right)\\
&>2\left(|\nabla A|^2-\left(\mathring{a}^\prime+2|H|^2\mathring{a}^{\prime\prime}\right)|\nabla H|^2\right)-2R_1+2\left(\frac{1}{n}+\mathring{a}^\prime\right)R_2+2n\bar{K}|A|^2\\
&\ge2\left(|\nabla A|^2-\frac{2(n-1)}{n(n+2)}|\nabla H|^2\right)\\
&+2\left(\left(\frac{1}{n}+\mathring{a}^\prime\right)\sum_{i,j}|\langle A_{ij},H\rangle|^2-\sum_{i,j,p,q}|\langle A_{ij},A_{pq}\rangle|^2-\sum_{i,j}|R_{ij}^\bot |^2\right)
	\end{align*}
and so, from Lemma \ref{inequalities} (ii) and Kato inequality \eqref{katoinequality}, we have
	\begin{align}\label{evolutionequation}
	&\left(\partial_t-\Delta\right)\frac{|A^-|^2}{f}=\frac{1}{f}\left(\partial_t-\Delta\right)|A^-|^2-|A^-|^2\frac{1}{f^2}\left(\partial_t-\Delta\right)f+2\left\langle\nabla\frac{|A^-|^2}{f},\nabla \log f\right\rangle\nonumber\\
	&<\frac{2}{f}\left(\sum_{i,j,p,q}|\langle A_{ij}^- ,A_{pq}^- \rangle|^2+\sum_{i,j,p}|h_{ip} A_{pj}^- -h_{jp}A_{ip}^- |^2+\sum_{i,j,p} |A_{ip}^- \otimes A_{jp}^- -A_{jp}^- \otimes A_{ip}^- |^2-n\bar{K}|A^-|^2\right)\nonumber\\
	&+\frac{2}{f}\left(-|\nabla^\bot A^-|^2+2\sum_{i,j,k}(\nabla_k h_{ij}-|H|^{-1}h_{ij} \nabla_k |H|)\langle A_{ij}^- ,\nabla^\bot_k \nu_1\rangle\right)\nonumber\\
	&-|A^-|^2\frac{2}{f^2}\left(\left(\frac{1}{n}+\mathring{a}^\prime\right)\sum_{i,j}|\langle A_{ij},H\rangle|^2-\sum_{i,j,p,q}|\langle A_{ij},A_{pq}\rangle|^2-\sum_{i,j}|R_{ij}^\bot |^2\right)\nonumber\\
	&-|A^-|^2\frac{2}{f^2}\left(|\nabla A|^2-\frac{2(n-1)}{n(n+2)}|\nabla H|^2\right)+2\left\langle\nabla\frac{|A^-|^2}{f},\nabla\log f\right\rangle.
	\end{align}

\subsection{Codimension estimate}
The goal for this subsection is to prove Theorem \ref{blowuptheorem}.

\begin{theorem}[Codimension estimate, cf.\cite{HNAV}, Theorem 5.12]\label{blowuptheorem}
Let $F: \mathcal{M}^n\times[0, T) \rightarrow \mathbb{S}^{n+m}\left( \frac{1}{\sqrt{\bar{K}}}\right)$ be a smooth solution to the mean curvature flow, so that $F_0(p)=F(p, 0)$ is compact and quartically pinched, where $n\ge8,m\ge2$.
Then, $\forall \e>0, \exists H_0 >0$, such that if $f \geq H_0$, then
	\begin{align*}
	|A^-|^2 \leq \e f+C_{\e},
	\end{align*}
$\forall t \in[0, T)$, where $C_\e=C_{\e}(n, m)$.
\end{theorem}
\begin{proof}
Since $\mathcal{M}$ is quartically bounded and from \eqref{ineqa}, there exist constants $ C,D$, such that $|A^-|^2 \leq Cf +D$. Therefore, the above estimate holds for all $\e\geq \frac{c_n}{\delta}$. Indeed, from the \eqref{ineqa}, we have $|A|^2\le a<c_n|H|^2+2\sqrt{n}\bar{K}$ and we can make a little bit more space so that
	\begin{align*}
	|A^-|^2+|A^+|^2=|A|^2\le (c_n-\delta)|H|^2-C_\delta+2\sqrt{n}\bar{K}
	\end{align*}
and therefore,
	\begin{align*}
	\delta|H|^2\le c_n|H|^2-|A|^2+2\sqrt{n}\bar{K}.
	\end{align*}
But since $|A^-|^2\le|A|^2<c_n|H|^2+2\sqrt{n}\bar{K}$, we have $\frac{\delta |A^-|^2}{c_n}\le\frac{\delta|A|^2}{c_n}<\delta|H|^2+\frac{\delta}{c_n}2\sqrt{n}\bar{K}$, so from \eqref{ineqa}, we have
	\begin{align*}
	\frac{\delta}{c_n} |A^-|^2 &< c_n |H|^2+\left(1+\frac{\delta}{c_n}\right)2\sqrt{n}\bar{K} -|A|^2\\
	&<-|A|^2+a+\left(\left(1+\frac{\delta}{c_n}\right)2\sqrt{n} -4\right)\bar{K}\\
	&=f+\left(\left(1+\frac{\delta}{c_n}\right)2\sqrt{n} -4\right)\bar{K},
	\end{align*}
which means that 
	\begin{align*}
	|A^-|^2\le\frac{c_n}{\delta}f+\frac{c_n}{\delta}\left(\left(1+\frac{\delta}{c_n}\right)2\sqrt{n} -4\right)\bar{K}\le\e f+C_\e.
	\end{align*}
Hence, let $\e_0$ denote the infimum of such $\e$ for which the estimate is true and suppose $\e_0>0$. We will prove the theorem by contradiction. Hence, let us assume that the conclusions of the theorem are not true. That is, there exists a family of mean curvature flow $\mathcal{M}_t^k$ with points $(p_k,t_k)$ such that
	\begin{align}\label{eqn_limitepsilon0}
	\lim_{k\rightarrow \infty} \frac{|A_k^-(p_k,t_k)|^2}{f_k(p_k,t_k)}= \e_0,
	\end{align}
with $\e_0>0$ and $ f_k(p_k,t_k)\rightarrow \infty$.
We perform a parabolic rescaling of $ \mathcal{M}_t^k $ in such a way that $f_k$ at $(p_k,t_k)$ becomes $1$. If we consider the exponential map $\exp_{\bar{p}}\colon T_{\bar{p}}\mathbb{S} \cong \mathbb{R}^{n+m}\to \mathbb{S}^{n+m}$ and $\gamma$ a geodesic, then for a vector $v\in T_{\bar{p}}\mathbb{S}$, then
	\begin{align*}
	\exp_{\bar{p}}(v)=\gamma_{\bar{p},\frac{v}{|v|}} (|v|), \ \ \gamma '(0)=\frac{v}{|v|} \ \ \text{ and} \ \ \gamma(0)=\bar{p}=F_k(p_k,t_k).
	\end{align*}
That is, if $F_k$ is the parameterisation of the original flow $ \mathcal{M}_t^k $, we let $ \hat r_k = \frac{1}{f_k(p_k,t_k)}$, and we denote the rescaled flow by $ \overline{\mathcal{M}}_\tau^k $ and we define its parameterisation by
	\begin{align*}
	\overline F_k (p,\tau) = \exp^{-1}_{F_k(p_k,t_k)} \circ F_k (p,\hat{r}^2_k \tau+t_k).
	\end{align*}
In the Riemannian case, when we change the metric after dilation, we do not need to multiply the immersion by the same constant as we would do in the Euclidean space. When we rescale the background space, we see that
	\begin{align*}
	\bar{g}_{ij}=\frac{1}{\hat{r}^2_k}g_{ij} \ \ \text{ and} \ \ \overline{K}=\hat{r}^2_k \bar{K},
	\end{align*}
where $\bar{K}$ is the sectional curvature of $\mathbb{S}^{n+m}$. In the same way,
	\begin{align*}
	|\overline{A}|^2=\hat{r}^2_k |A|^2 \ \ \text{ and} \ \ |\overline{H}|^2=\hat{r}^2_k |H|^2.
	\end{align*}
Since $d_n$ depends on $n$ and the sectional curvature $\bar{K}$, the new $\bar{d}_n$ depends on n and $\overline{K}$. Hence,
	\begin{align*}
	\bar{d}_n=\hat{r}^2_k d_n.
	\end{align*}
For $\hat{r}_k\to 0$, the background Riemannian manifold will converge to its tangent plane in a pointed $C^{d,\gamma}$ H\"older topology \cite{Petersen2016}. Therefore, we can work on the manifold $\mathbb{S}^{n+m}$ as we would work in a Euclidean space. For simplicity, we choose for every flow a local co-ordinate system centred at $ p_k$. In these co-ordinates we can write $0$ instead of $ p_k$. Recall \eqref{parneigh}. The parabolic neighbourhoods $\mathcal P^k ( p_k, t_k, \hat r_k L, \hat r_k^2 \theta)$ in the original flow becomes $ \overline{\mathcal P}^k(0,0,L,\theta)$. By construction, each rescaled flow satisfies
	\begin{align} \label{eqn_H1}
	\overline F_k (0,0) = 0, \quad \overline f_k (0,0) = 1.
	\end{align}
Indeed,
	\begin{align*}
	\overline F_k(0,0)&=\exp_{F_k(0,0)}^{-1}\circ F_k(0,\hat r_k^2 \cdot 0)=0,\\
	\overline f_k (p,\tau)&=-|\overline A_k (p,\tau)|^2+\overline{a}_k\\
	&=-|\overline A_k (p,\tau)|^2+\sqrt{\left(\frac{|\overline{H}|^2}{n-2}+4 \overline{K}\right)^2+(4n-16) \overline{K}^2}\\
	&=\hat r_k^2\Big(-|A_k(p,\hat r_k^2 \tau+t_k)|^2+\sqrt{\left(\frac{|H_k(p,\hat r_k^2 \tau+t_k)|^2}{n-2}+4 \bar{K}\right)^2+(4n-16) \bar{K}^2}\Big)\\
	&=\hat r_k^2\Big(-|A_k(p,\hat r_k^2 \tau+t_k)|^2+a_k\Big)\\
	&=\hat r_k^2 f_k(p,\hat r_k^2 \tau+t_k).
	\end{align*}

and so
	\begin{align*}
	\overline f_k (0,0)=\hat r_k^2 f_k(0,0)=1,
	\end{align*}
since $\hat r_k (0,0)=\frac{1}{f_k (0,0)}=1$ from the change of coordinates. The gradient estimates give us uniform bounds (depending only on the pinching constant) on $ |A_k|$ and its derivatives up to any order on a neighbourhood of the form $\overline{\mathcal P }^k( 0 ,0,d,d)$ for a suitable $ d > 0$. From Theorem \ref{thm_gradient}, we obtain gradient estimates on the second fundamental form in $ C^\infty $ on $ \overline F_k$. Hence we can apply Arzela-Ascoli (via the Langer-Breuning compactness theorem \cite{Breuning2015} and \cite{Langer1985}) and conclude there exists a subsequence converging in $ C^\infty $ to some limit flow which we denote by $ \widetilde{\mathcal{M}}_\tau^\infty$. We analyse the limit flow $ \widetilde{\mathcal{M}}_\tau^\infty$. Note we have for the Weingarten map
	\begin{align*}
		[\overline A_k^-]_i^j ( p , \tau) = \hat r_k [A_k^-]_i^j ( p , \hat r_k^2 \tau+t_k),
	\end{align*}
so that
	\begin{align*}
	\frac{|\overline A_k^-(p, \tau) |^2}{\overline{f}_k(p,\tau)}&= \frac{|A_k^- ( p, \hat r_k^2 \tau+t_k) |^2}{f_k(p, \hat r_k^2 \tau+t_k)}.
	\end{align*}
From \eqref{eqn_limitepsilon0} and \eqref{eqn_H1}, we see
	\begin{align*}
	\frac{| \widetilde A^-(0,0)|^2}{\widetilde f(0,0)}= \e_0, \quad \widetilde f(0,0) = 1.
	\end{align*}
We claim
	\begin{align*}
	\frac{|\widetilde A^- ( p , \tau) |^2}{\widetilde f (p, \tau)}=\lim_{k\rightarrow \infty}\frac{|\overline A_k^-(p,\tau)|^2}{\overline f_k (p,\tau)} \leq \e, \quad \forall \e>\e_0.
	\end{align*}
 Since $\widetilde{f}(0,0)=1$, it follows that $|\widetilde{f}| \geq \frac{1}{2}$ in $ \widetilde{\mathcal{P}}^\infty(0,0,r,r)$ for some $r < d^\#$, where $d^\#$ is the constant from [\cite{Nguyen2018}, Lemma 3.5].
This is true since any point $( p, \tau) \in \widetilde{\mathcal{M}}^\infty_\tau$ is the limit of points $(p_{j_k},t_{j_k}) \in \overline{\mathcal{M}}^k_\tau$ and for every $ \e > \e_0 $ if we let $ \eta = \eta(\e,c_n)< d^\#$ then for large $k$, $\mathcal{M}^k_t$ is defined in
	\begin{align*}
	 \mathcal P^k\left(p_{j_k},t_{j_k}, \frac{1}{f_k(p_{j_k},t_{j_k})}\eta,\left(\frac{1}{f_k(p_{j_k},t_{j_k})}\right)^2 \eta\right),
	\end{align*}
which implies
	\begin{align*}
	\frac{|\overline A_k^- ( p_{j_k} , t_{j_k}) |^2}{\overline f_k (p_{j_k}, t_{j_k})} \leq \e, \quad \forall \e>\e_0.
	\end{align*}
Hence the flow $\widetilde{\mathcal{M} }_\tau^\infty \subset \mathbb R^{n+m}$ has a space-time maximum $\e_0$ for $\frac{|\widetilde A^- ( p , \tau) |^2}{\widetilde f (p, \tau)}$ at $ (0,0)$. 
But in the limit our background space is Euclidean. Therefore, from [\cite{Naff}, section 4.3], the evolution equation becomes
	\begin{align*}
	\left(\partial_t-\Delta\right)\frac{|\widetilde A^-|^2}{\widetilde f}&\le2\left\langle\nabla\frac{|\widetilde A^-|^2}{\widetilde f},\nabla\log \widetilde f\right\rangle-\delta\frac{|\widetilde A^-|^2}{\widetilde f^2}\left(\partial_t-\Delta\right)\widetilde f.
	\end{align*}
Hence, since $\frac{|\widetilde A^-|^2}{\widetilde f}$ attains a maximum $\e_0$ at $(0,0)$ by the strong maximum principle, then there exists this constant $\mathcal{C}$ depending up to $n, \bar{K}$, such that $\frac{|\widetilde A^-|^2}{\widetilde f}=\mathcal{C}.$ Putting this into the evolution equation, we have
	\begin{align*}
	0&\le-\delta\frac{\mathcal{C}}{\widetilde f}\left(\partial_t-\Delta\right)\widetilde f\le 0,
	\end{align*}
which means that we get $\mathcal{C}=0$ and therefore, $|\widetilde A^-|=0$.
This implies
	\begin{align*}
	\frac{|\widetilde A^-|^2}{\widetilde{f}}=0\implies \e_0=0,
	\end{align*}which is a contradiction. Hence, we obtain
	\begin{align*}
	\lim_{k\rightarrow \infty} \frac{|\overline A_k^-(p_k,t_k)|^2}{\overline f_k(p_k,t_k)} = 0.
	\end{align*}
\end{proof}
\subsection{Cylindrical estimate}
Here, we present estimates that demonstrate an improvement in curvature as we approach a singularity. These estimates play a critical role in the analysis of high curvature regions in geometric flows. In particular, in the high codimension setting, we prove that the pinching ratio $\frac{|A|^2}{|H|^2}$ approaches the ratio of the standard cylinder, which is $\frac{1}{n-1}$.
\begin{theorem}[Cylindrical estimate, cf.\cite{HuSi09}, cf.\cite{HNAV}]\label{thm_cylindrical}
Let $F: \mathcal{M}^n\times[0, T) \rightarrow \mathbb{S}^{n+m}\left(\frac{1}{\sqrt{\bar{K}}}\right)$ be a smooth solution to the mean curvature flow, so that
$F_0(p)=F(p, 0)$ is compact and quartically pinched with constant $c_n=\frac{1}{n-2}$, where $n\ge8$.
Then $\forall \e>0, \exists H_1 >0$, such that if $|H|\geq H_1$, then
	\begin{align*}
	|A|^2- \frac{1}{n-1}|H|^2\leq \e |H|^2+C_{\e},
	\end{align*}
$\forall t \in[0, T)$, where $C_\e=C_{\e}(n)$.
\end{theorem}
\begin{proof}
The proof follows closely the proof of Theorem \ref{blowuptheorem}. From \eqref{ineqa}, we have
	\begin{align*}
	&|A|^2\le a<\frac{1}{n-2}|H|^2+2\sqrt{n}\bar{K}
	\end{align*}
and therefore, there exist constants $ C,D$ such that
	\begin{align*}
	|A|^2- \frac{1}{n-1}|H|^2 \leq C|H|^2 +D.
	\end{align*}
Hence, let $\e_0$ denote the infimum of such $\e$, for which the estimate is true and suppose $\e_0>0$. We will prove the theorem by contradiction. Hence, let us assume that the conclusions of the theorem are not true. That is, there exists a family of mean curvature flow $\mathcal{M}_t^k$ with points $(p_k,t_k)$, such that
	\begin{align}\label{eqn_limitepsilon0versiontwo}
	\lim_{k\rightarrow \infty} \frac{\left (|A_k(p_k,t_k)|^2- \frac{1}{n-1}|H_k(p_k,t_k)|^2\right )}{|H_k(p_k,t_k)|^2}= \e_0,
	\end{align}
with $\e_0>0$ and $|H_k(p_k,t_k)|^2\rightarrow \infty$. We perform a parabolic rescaling of $ \mathcal{M}_t^k $ exactly as in Theorem \ref{blowuptheorem}, which is in such a way that $|H_k|^2$ at $(p_k,t_k)$ becomes $1$. If we consider the exponential map $\exp_{\bar{p}}\colon T_{\bar{p}}\mathbb{S} \cong \mathbb{R}^{n+m}\to \mathbb{S}^{n+m}$ and $\gamma$ a geodesic, then for a vector $v\in T_{\bar{p}}\mathbb{S}$, we have
	\begin{align*}
	\exp_{\bar{p}}(v)=\gamma_{\bar{p},\frac{v}{|v|}} (|v|), \ \ \gamma '(0)=\frac{v}{|v|} \ \ \text{ and} \ \ \gamma(0)=\bar{p}=F_k(p_k,t_k).
	\end{align*}
That is, if $F_k$ is the parameterisation of the original flow $ \mathcal{M}_t^k $, we let $ \hat r_k = \frac{1}{|H_k(p_k,t_k)|^2}$, and we denote the rescaled flow by $ \overline{\mathcal{M}}_\tau^k $ and we define its parameterisation by
	\begin{align*}
	\overline F_k (p,\tau) = \exp^{-1}_{F_k(p_k,t_k)} \circ F_k (p,\hat{r}^2_k \tau+t_k).
	\end{align*}
For $\hat{r}_k\to 0$, the background Riemannian manifold will converge to its tangent plane in a pointed $C^{d,\gamma}$ H\"older topology \cite{Petersen2016}. Therefore, we can work on the manifold $\mathbb{S}$ as we would work in a Euclidean space. For simplicity, we choose for every flow a local co-ordinate system centred at $ p_k$. In these co-ordinates we can write $0$ instead of $ p_k$. The parabolic neighbourhoods $\mathcal P^k ( p_k, t_k, \hat r_k L, \hat r_k^2 \theta)$ in the original flow becomes $ \overline{\mathcal P}^k(0,0,L,\theta)$. By construction, each rescaled flow satisfies
	\begin{align} \label{eqn_H1version2}
	\overline F_k (0,0) = 0, \quad |\overline H_k (0,0)|^2 = 1.
	\end{align}
The gradient estimates give us uniform bounds (depending only on the pinching constant) on $ |A_k|$ and its derivatives up to any order on a neighbourhood of the form $\overline{\mathcal P }^k( 0 ,0,d,d)$ for a suitable $ d > 0$. From Theorem \eqref{thm_gradient}, we obtain gradient estimates on the second fundamental form in $ C^\infty $ on $ \overline F_k$. Hence we can apply Arzela-Ascoli (via the Langer-Breuning compactness theorem \cite{Breuning2015} and \cite{Langer1985}) and conclude there exists a subsequence converging in $ C^\infty $ to some limit flow which we denote by $ \widetilde{\mathcal{M}}_\tau^\infty$. 
From \eqref{eqn_limitepsilon0versiontwo} and \eqref{eqn_H1version2}, we see
	\begin{align*}
	\frac{| \widetilde A(0,0)|^2-\frac{1}{n-1}|\widetilde H(0,0)|^2}{|\widetilde H(0,0)|^2}= \e_0, \quad |\widetilde H(0,0)|^2= 1.
	\end{align*}
We claim
	\begin{align*}
	\frac{|\widetilde A( p , \tau) |^2-\frac{1}{n-1}|\widetilde H ( p , \tau) |^2}{|\widetilde H (p, \tau)|^2}=\lim_{k\rightarrow \infty}\frac{|\overline A_k(p,\tau)|^2-\frac{1}{n-1}|\overline H_k(p,\tau)|^2}{|\overline H_k (p,\tau)|^2} \leq \e, \quad \forall \e>\e_0.
	\end{align*}
 Since $|\widetilde{H}(0,0)|^2=1$, it follows that $|\widetilde{H}|^2 \geq \frac{1}{2}$ in $ \widetilde{\mathcal{P}}^\infty(0,0,r,r)$ for some $r < d^\#$, where $d^\#$ is the constant from [\cite{Nguyen2018}, Lemma 3.5].
This is true since any point $( p, \tau) \in \widetilde{\mathcal{M}}^\infty_\tau$ is the limit of points $(p_{j_k},t_{j_k}) \in \overline{\mathcal{M}}^k_\tau$ and for every $ \e > \e_0 $ if we let $ \eta = \eta(\e,c_n)< d^\#$ then for large $k$, $\mathcal{M}^k_t$ is defined in
	\begin{align*}
	 \mathcal P^k\left(p_{j_k},t_{j_k}, \frac{1}{|H_k(p_{j_k},t_{j_k})|^2}\eta,\left(\frac{1}{|H_k(p_{j_k},t_{j_k})|^2}\right)^2 \eta\right),
	\end{align*}
which implies
	\begin{align*}
	\frac{|\overline A_k ( p_{j_k} , t_{j_k}) |^2-\frac{1}{n-1}|\overline H_k( p_{j_k} , t_{j_k}) |^2}{|\overline H_k (p_{j_k}, t_{j_k})|^2} \leq \e, \quad \forall \e>\e_0.
	\end{align*}
Hence, the flow $\widetilde{\mathcal{M} }_\tau^\infty \subset \mathbb R^{n+m}$ has a space-time maximum $\e_0$ for $\frac{|\widetilde A( p , \tau) |^2-\frac{1}{n-1}|\widetilde H ( p , \tau) |^2}{|\widetilde H (p, \tau)|^2}$ at $ (0,0)$, which implies that the flow $\widetilde{\mathcal M }_\tau^\infty$ has a space-time maximum $\frac{1}{n-1}+\e_0$ for $\frac{|\widetilde A ( p , \tau) |^2}{| \widetilde H (p, \tau) |^2}$ at $ (0,0)$. Since the evolution equation for $ \frac{|A|^2}{|H|^2}$ is given by
	\begin{align*}
	\Big(\partial_t -\Delta\Big) \frac{|A|^2}{|H|^2}&= \frac{2}{|H|^2}\left\la \nabla |H|^2 , \nabla \left( \frac{|A|^2}{|H|^2}\right) \right\ra-\frac{2}{|H|^2} \left( |\nabla A|^2-\frac{|A|^2}{|H|^2}|\nabla H|^2 \right) \\
	&+\frac{2}{|H|^2}\left( R_1-\frac{|A|^2}{|H|^2} R_2\right)+4\bar{K}\left(1-n\frac{|A|^2}{|H|^2}\right)
	\end{align*}
and knowing that $\frac{3}{n+2}|\nabla H|^2 \leq |\nabla A|^2$ and $\frac{|A|^2}{|H|^2}\leq c_n,$ we arrive at
	\begin{align*}
	-\frac{2}{|H|^2} \left( |\nabla A|^2-\frac{|A|^2}{|H|^2}|\nabla H|^2 \right) \leq 0 \ \ \text{and} \ \ 4\bar{K}\left(1-n\frac{|A|^2}{|H|^2}\right)<0.
	\end{align*}
Furthermore, if $\frac{|A|^2}{|H|^2}=c < c_n$, for $\varepsilon_0$ small enough, according to Lemma 2.3 in \cite{HTNsurgery} and Proposition \ref{preservation}, we have
	\begin{align*}
	 R_1&-\frac{|A|^2}{|H|^2} R_2= R_1-c R_2\\
	 &\leq \frac{2}{n}\frac{1}{c-\nicefrac{1}{n}}| A^-|^2f+\left(6-\frac{2}{n (c-\nicefrac{1}{n})} \right) |\circo h|^2 | \circo A^-|^2+\left(3-\frac{2}{n (c-\nicefrac{1}{n})} \right)|\circo A^-|^4+2f|h|^2\\
	&\leq 0.
	\end{align*}
Hence, the strong maximum principle applies to the evolution equation of $\frac{|A|^2}{|H|^2}$ and shows that $\frac{|A|^2}{|H|^2}$ is constant. The evolution equation then shows $ |\nabla A|^2= 0$, that is the second fundamental form is parallel and that $|A^-|^2 = |\circo A^-|^2=0$, that is the submanifold is codimension one. Finally, this shows locally $ \mathcal{M} = \mathbb S^{n-q}\times \mathbb R^q$, \cite{Lawson1969}. As $\frac{|A|^2}{|H|^2}< c_n\leq \frac{1}{n-2}$, we can only have
	\begin{align*}
	\mathbb S^n, \mathbb S^{n-1}\times \mathbb R,
	\end{align*}
which gives $\frac{|A|^2}{|H|^2}= \frac{1}{n}, \frac{1}{n-1}\neq \frac{1}{n-1}+\e_0, \e_0>0$, which gives a contradiction.
\end{proof}
\section{Singularity models of quartically pinched solutions of mean curvature flow in higher codimension}
In this subsection, we derive a corollary from Theorem \ref{blowuptheorem}, which provides information about the blow up models at the first singular time. Specifically, we show that these models can be classified up to homothety. 
\begin{theorem}[{\cite[Theorem 1.1]{Naff6}}]\label{naff1.1} Suppose $n \geq 3$ and let $F_t: M \rightarrow \mathbb{R}^{n+1}, t \in(-\infty, 0]$, be an $n$-dimensional, complete, connected ancient solution to the mean curvature flow in $\mathbb{R}^{n+1}$ that is weakly convex, uniformly twoconvex, and satisfies pointwise derivative estimates: $|\nabla A| \leq \gamma_1|H|^2$ and $\left|\nabla^2 A\right| \leq \gamma_2|H|^3$. Then the solution is noncollapsed.
\end{theorem}
As a consequence of the above, he obtains the following classification.
\begin{corollary}[{\cite[Corollary 1.4]{Naff6}}]\label{resultnaff} Let $n \geq 5$ and $N>n$. Let $c_n=\frac{1}{n-2}$ if $n \geq 8$ and $c_n=\frac{3(n+1)}{2 n(n+2)}$ if $n=5,6$, or 7 . Consider a closed, n-dimensional solution to the mean curvature flow in $\mathbb{R}^N$ initially satisfying $|H|>0$ and $|A|^2<c_n|H|^2$. At the first singular time, the only possible blow-up limits are codimension one shrinking round spheres, shrinking round cylinders and translating bowl solitons.
\end{corollary}
Let $\widetilde{M}_\tau$ be any blow-up limit at the first singular time. By Theorem \ref{blowuptheorem}, the limit is a codimension-one ancient flow in $\mathbb{R}^{n+1}$. The cylindrical estimate in Theorem \ref{thm_cylindrical} improves under blow-up to the sharp hypersurface pinching $|A|^2 \leq \frac{1}{n-1}|H|^2$, hence $\widetilde{M}_\tau$ is weakly convex and uniformly two-convex. Moreover, the scale-invariant gradient estimates in section \ref{gradientestimates} pass to the limit and, using the quartic pinching $f=-|\mathring{A}|^2+\mathring{a}-\varepsilon \omega>0$, yield the derivative bounds $|\nabla A| \leq \gamma_1|H|^2,\left|\nabla^2 A\right| \leq \gamma_2|H|^3$. Thus, $\widetilde{M}_\tau$ satisfies the hypotheses of Theorem \ref{naff1.1} and is therefore noncollapsed. Hence, we obtain the following classification as in Corollary \ref{resultnaff}.
\begin{corollary}[cf.\cite{HS}, cf.\cite{HNAV}, cf.\cite{AVCPn}]
Let $c_n=\frac{1}{n-2}$. Suppose $F_t\colon\mathcal{M}^n \rightarrow \mathbb{S}^{n+m}$, $n\ge8,$ $m\ge 2$ is a smooth solution of the mean curvature flow, compact and quartically pinched with $|H|>0$, on the maximal time interval $[0, T)$.
\begin{enumerate}
\item If the singularity for $t \rightarrow T$ is of type I, the only possible limiting flows are the homothetically shrinking solutions associated with $\mathbb{S}^n, \mathbb{R} \times \mathbb{S}^{n-1}$.
\item If the singularity is of type II, then from Theorem \ref{blowuptheorem} and Theorem \ref{thm_gradient}, the only possible blow-up limits at the first singular time are codimension one shrinking round spheres, shrinking round cylinders, and translating bowl solitons.
\end{enumerate}
\end{corollary}

\section{Convergence for infinite time}
In this section, we assume $T= \infty$. We primarily follow \cite{AVCPn}. We prove a decay estimate on the traceless part of the second fundamental form, which proves that the submanifold $\mathcal{M}_t$ converges to a totally geodesic limit, as $t\to T$.
\begin{proposition}[Decay estimate, cf.\cite{AVCPn}]\label{theoreminfinitetime}
Let $F: \mathcal{M}^n\times[0, T) \rightarrow \mathbb{S}^{n+m}\left(\frac{1}{\sqrt{\bar{K}}}\right)$ be a smooth solution to the mean curvature flow, so that $F_0(p)=F(p, 0)$ is compact and quartically pinched, $n\ge8$.
Then, there exists a positive constant $\mathcal{C}=\mathcal{C}_\varepsilon(n, \bar{K})$, depending only on the initial manifold $M_0$, such that
	\begin{align*}
	\frac{|\mathring{A}|^2}{f}\le\mathcal{C}e^{-2\left(2\sqrt{n}\bar{K}(1-2\varepsilon n)\right)t},
	\end{align*}
for any $0\le t<T=\infty$.
\end{proposition}
\begin{proof}
Consider the following functions
	\begin{align*}
Q :=-\frac{f}{2} \ \ \text{and} \ \ q:=\frac{1}{2}\left(|A|^2-\frac{1}{n}|H|^2\right),
	\end{align*}
where recall $f:=-|A|^2+a-\varepsilon\omega$, where $a,\omega$ come from \eqref{a(x)} and \eqref{omega}. In the case where $H\neq 0$, from \eqref{eqn_|A|^2}, \eqref{eqn_|H|^2} and as we did in Proposition \ref{preservation}, the evolution equation of $Q$ becomes
	\begin{align*}
 \left(\partial_t-\Delta\right) Q&\le -\left(|\nabla A|^2-\frac{2(n-1)}{n(n+2)}|\nabla H|^2\right)-n\bar{K}|A|^2\\
&-\left(\left(\frac{1}{n}+\mathring{a}^\prime\right)\sum_{i,j}|\langle A_{ij},H\rangle|^2-\sum_{i,j,p,q}|\langle A_{ij},A_{pq}\rangle|^2-\sum_{i,j}|R_{ij}^\bot |^2\right).
	\end{align*}
At a point where $H \neq 0$, decomposing $A$ into its irreducible components according to \cite{AnBa10}, \cite{LaLyNg}, \cite{Li1992} and \cite{Naff}, we have
	\begin{align*}
& |A|^2=|\mathring{h}|^2+\frac{1}{n} |H|^2+|A^-|^2, \\
&\sum_{i,j}|\langle A_{ij},H\rangle|^2=|\mathring{h}|^2 |H|^2+\frac{1}{n} |H|^4, \\
&\sum_{i,j,p,q} |\langle A_{ij}, A_{pq}\rangle |^2+\sum_{i,j}|R^{\bot} _{ij}|^2\le 3|\mathring{h}|^2|A^-|^2+\frac{3}{2}|A^-|^4+\left(|\mathring{h}|^2+\frac{1}{n}|H|^2\right)|A|^2-\frac{1}{n}|A^-|^2|H|^2
	\end{align*}
and so
	\begin{align*}
	\left(\partial_t-\Delta\right)Q&\le 3|\mathring{h}|^2|A^-|^2+\frac{3}{2}|A^-|^4+\left(|\mathring{h}|^2+\frac{1}{n}|H|^2\right)|A|^2-\frac{1}{n}|A^-|^2|H|^2-n\bar{K}|A|^2\\
&-\left(\frac{1}{n}+\mathring{a}^\prime\right)\left(|\mathring{h}|^2+\frac{1}{n}|H|^2\right) |H|^2-\left(|\nabla A|^2-\frac{2(n-1)}{n(n+2)}|\nabla H|^2\right).
	\end{align*}
Also, from \eqref{ineqa}, \eqref{omega}, 
we have
	\begin{align}
\label{form1}	&|A|^2=2Q+a-\varepsilon\omega<2Q+\frac{1-\varepsilon}{n-2}|H|^2+2\sqrt{n}\bar{K}(1-2\varepsilon n),\\
\label{form2}	&|A|^2=2Q+a-\varepsilon\omega>2Q+\frac{1-\varepsilon}{n-2}|H|^2+4\bar{K}(1-\varepsilon n\sqrt{n}).
	\end{align}
Moreover, from \eqref{katoinequality}, we have that the term $|\nabla A|^2-\frac{2(n-1)}{n(n+2)}|\nabla H|^2$ is negative. Therefore,
	\begin{align*}
&\left(\partial_t-\Delta\right)Q\le3|\mathring{h}|^2|A^-|^2+\frac{3}{2}|A^-|^4-\frac{1}{n}|A^-|^2|H|^2+\left(|\mathring{h}|^2+\frac{1}{n}|H|^2\right)(2Q+a-\varepsilon\omega)\\
&-n\bar{K}(|\mathring{h}|^2+|A^-|^2)-\bar{K}|H|^2-\left(\frac{1}{n}+\mathring{a}^\prime\right)\left(|\mathring{h}|^2+\frac{1}{n}|H|^2\right)|H|^2\\
&<3|\mathring{h}|^2|A^-|^2+\frac{3}{2}|A^-|^4-\frac{1}{n}|A^-|^2|H|^2+\left(|\mathring{h}|^2+\frac{1}{n}|H|^2\right)(2Q+2\sqrt{n}\bar{K}(1-2\varepsilon n))-\bar{K}|H|^2\\
&-n\bar{K}(|\mathring{h}|^2+|A^-|^2)+\left(\frac{2-n\varepsilon}{n(n-2)}-\mathring{a}^\prime\right)|h|^2|H|^2\\
&=\left(3|\mathring{h}|^2+\frac{3}{2}|A^-|^2-\frac{1}{n}|H|^2-2\sqrt{n}\bar{K}(1-2\varepsilon n)\right)|A^-|^2-n\bar{K}(|\mathring{h}|^2+|A^-|^2)-\bar{K}|H|^2\\
&+2\sqrt{n}\bar{K}(1-2\varepsilon n)\left(|\mathring{h}|^2+|A^-|^2+\frac{1}{n}|H|^2\right)+2Q\left(|\mathring{h}|^2+\frac{1}{n}|H|^2\right)+\left(\frac{2-n\varepsilon}{n(n-2)}-\mathring{a}^\prime\right)|h|^2|H|^2.
	\end{align*}
Substituting \eqref{form1}, we obtain
	\begin{align*}
2\sqrt{n}\bar{K}(1-2\varepsilon n)&\left(|\mathring{h}|^2+|A^-|^2+\frac{1}{n}|H|^2\right)\\
&<2\sqrt{n}\bar{K}(1-2\varepsilon n)\left(2Q+\frac{1-\varepsilon}{n-2}|H|^2+2\sqrt{n}\bar{K}(1-2\varepsilon n)\right)
	\end{align*}
and substituting \eqref{form2}, we have
	\begin{align*}
-n\bar{K}|\mathring{A}|^2&=-n\bar{K}(|\mathring{h}|^2+|A^-|^2)<-n\bar{K}\left(2Q+\left(\frac{1-\varepsilon}{n-2}-\frac{1}{n}\right)|H|^2+4\bar{K}(1-\varepsilon n\sqrt{n})\right)
	\end{align*}
and hence
	\begin{align*}
&2\sqrt{n}\bar{K}(1-2\varepsilon n)\left(|\mathring{h}|^2+|A^-|^2+\frac{1}{n}|H|^2\right)-n\bar{K}|\mathring{A}|^2\\
& <2Q\left(\frac{2}{\sqrt{n}}(1-2\varepsilon n)-1\right)n\bar{K}+|H|^2\left(\frac{2}{\sqrt{n}}(1-2\varepsilon n)\left(\frac{1-\varepsilon}{n-2}\right)-\left(\frac{1-\varepsilon}{n-2}-\frac{1}{n}\right)\right)n\bar{K}\\
&+n^2\bar{K}^2\left(\frac{4}{n}(1-2\varepsilon n)^2-\frac{4}{n}(1-\varepsilon n\sqrt{n}) \right)\\
& <2Q\left(\frac{2}{\sqrt{n}}(1-2\varepsilon n)\right)n\bar{K}+|H|^2\left(\left(\frac{2-4\varepsilon n-\sqrt{n}}{\sqrt{n}}\right)\left(\frac{1-\varepsilon}{n-2}\right)+\frac{1}{n}\right)n\bar{K}\\
&+n^2\bar{K}^2\left(\frac{4}{n}(1-2\varepsilon n)^2-\frac{4}{n}(1-\varepsilon n\sqrt{n}) \right)\\
& <2Q\left(\frac{2}{\sqrt{n}}(1-2\varepsilon n)\right)n\bar{K}+|H|^2\bar{K}+n^2\bar{K}^2\left(\frac{4}{n}(1-2\varepsilon n)^2-\frac{4}{n}(1-\varepsilon n\sqrt{n}) \right).
	\end{align*}
Also, from \eqref{form1}, writing
	\begin{align*}
\frac{1}{n}|H|^2&>\frac{n-2}{2-\varepsilon n}\left(|\mathring{h}|^2+|A^-|^2-2Q-2\sqrt{n}\bar{K}(1-2\varepsilon n)\right),
	\end{align*}
we find
	\begin{align*}
&3|\mathring{h}|^2+\frac{3}{2}|A^-|^2-\frac{1}{n}|H|^2-2\sqrt{n}\bar{K}(1-2\varepsilon n)\\
&<3|\mathring{h}|^2+\frac{3}{2}|A^-|^2-\frac{n-2}{2-\varepsilon n}\left(|\mathring{h}|^2+|A^-|^2-2Q-2\sqrt{n}\bar{K}(1-2\varepsilon n)\right)-2\sqrt{n}\bar{K}(1-2\varepsilon n)\\
&=\left(3-\frac{n-2}{2-\varepsilon n}\right)|\mathring{h}|^2+\left(\frac{3}{2}-\frac{n-2}{2-\varepsilon n}\right)|A^-|^2+2Q\frac{n-2}{2-\varepsilon n}-\left(1-\frac{n-2}{2-\varepsilon n}\right)2\sqrt{n}\bar{K}(1-2\varepsilon n).
	\end{align*}
For $0<\varepsilon<1$, the first term on the right hand side is non positive. Indeed,
	\begin{align}\label{negativeterm}
	&3-\frac{n-2}{2-\varepsilon n}\le0\Longleftrightarrow n\ge\frac{8}{1+3\varepsilon},
	\end{align}
where for $0<\varepsilon<1$, we have $4<\frac{8}{1+3\varepsilon}<8$. Therefore, \eqref{negativeterm} holds for $n\ge5$. 
Disregarding this term and putting things back together we conclude that
	\begin{align*}
\left(\partial_t-\Delta\right)Q&< \left(\frac{3}{2}-\frac{n-2}{2-\varepsilon n}\right)|A^-|^4-\left(1-\frac{n-2}{2-\varepsilon n}\right)2\sqrt{n}\bar{K}(1-2\varepsilon n)|A^-|^2\\
&+n^2\bar{K}^2\left(\frac{4}{n}(1-2\varepsilon n)^2-\frac{4}{n}(1-\varepsilon n\sqrt{n}) \right)+\left(\frac{2-\varepsilon n}{n(n-2)}-\mathring{a}^\prime\right)|h|^2|H|^2\\
&+2Q\left(|\mathring{h}|^2+\frac{n-2}{2-\varepsilon n}|A^-|^2+\frac{1}{n}|H|^2+\left(\frac{2}{\sqrt{n}}(1-2\varepsilon n)\right)n\bar{K}\right).
	\end{align*}
Moreover, from \eqref{a'(x)}, we note that $\mathring{a}^\prime=\frac{\varepsilon_0}{n-2}-\frac{1}{n}$, where 
	\begin{align*}
	\varepsilon_0=\frac{\frac{|H|^2}{n-2}+4\bar{K}}{\sqrt{\left(\frac{|H|^2}{n-2}+4 \bar{K}\right)^2+(4 n-16) \bar{K}^2}}<1.
	\end{align*}
Denoting $\varepsilon_0:=1-\varepsilon$, we see that $\mathring{a}^\prime=\frac{2-\varepsilon n}{n(n-2)}$ and, hence, the term $\left(\frac{2-\varepsilon n}{n(n-2)}-\mathring{a}^\prime\right)|h|^2|H|^2$ is zero. Also, the discriminant of the polynomial
	\begin{align*}
&\left(\frac{3}{2}-\frac{n-2}{2-\varepsilon n}\right)|A^-|^4-\left(1-\frac{n-2}{2-\varepsilon n}\right)2\sqrt{n}\bar{K}(1-2\varepsilon n)|A^-|^2\\
&+n^2\bar{K}^2\left(\frac{4}{n}(1-2\varepsilon n)^2-\frac{4}{n}(1-\varepsilon n\sqrt{n}) \right)
	\end{align*}
is negative. Indeed,
	\begin{align*}
	D&=\left(1-\frac{n-2}{2-\varepsilon n}\right)^2\frac{4}{n}(1-2\varepsilon n)^2-4\left(\frac{3}{2}-\frac{n-2}{2-\varepsilon n}\right)\left(\frac{4}{n}(1-2\varepsilon n)^2-\frac{4}{n}(1-\varepsilon n\sqrt{n}) \right)\\
&=\left(-5+\frac{2(n-2)}{2-\varepsilon n}+\left(\frac{n-2}{2-\varepsilon n}\right)^2\right)(1-2\varepsilon n)^2-\left(6-\frac{4(n-2)}{2-\varepsilon n}\right)\left((1-2\varepsilon n)^2-(1-\varepsilon n\sqrt{n}) \right)\\
&=\left(-5+\frac{2(n-2)}{2-\varepsilon n}+\left(\frac{n-2}{2-\varepsilon n}\right)^2\right)(1-2\varepsilon n)^2+\left(6-\frac{4(n-2)}{2-\varepsilon n}\right)(1-\varepsilon n\sqrt{n})\\
&<0,	
	\end{align*}
for any $\varepsilon<1$ and any $n\ge0$, so we can disregard these terms. Therefore, whenever $H\neq 0$,  the evolution equation of $Q$, becomes
	\begin{align*}
\left(\partial_t-\Delta\right)Q&<2Q\left(|\mathring{h}|^2+\frac{n-2}{2-\varepsilon n}|A^-|^2+\frac{1}{n}|H|^2+\left(\frac{2}{\sqrt{n}}(1-2\varepsilon n)\right)n\bar{K}\right).
	\end{align*}
For the evolution equation of $q$, we have
	\begin{align*}
\left(\partial_t-\Delta\right)q&<\left(3|\mathring{h}|^2+\frac{3}{2}|A^-|^2-\frac{1}{n}|H|^2\right)|A^-|^2-n\bar{K}(|\mathring{h}|^2+|A^-|^2)+2q\left(|\mathring{h}|^2+\frac{1}{n}|H|^2\right).
	\end{align*}
Since $|\mathring{h}|^2+|A^-|^2=|\mathring{A}|^2=2q$, we have
	\begin{align*}
3|\mathring{h}|^2+\frac{3}{2}|A^-|^2-\frac{1}{n}|H|^2&=\left(3-\frac{n-2}{2-\varepsilon n}\right)|\mathring{h}|^2+\left(\frac{3}{2}-\frac{n-2}{2-\varepsilon n}\right)|A^-|^2-\frac{1}{n}|H|^2+\frac{n-2}{2-\varepsilon n}2q.
	\end{align*}
The first three terms of the above equality are non-positive for any $n\ge 8$ and $0<\varepsilon\ll\frac{1}{n^{5/2}}$, from \eqref{negativeterm}. Moreover,
	\begin{align*}
	-n\bar{K}(|\mathring{h}|^2+|A^-|^2)=-n\bar{K}|\mathring{A}|^2=-n\bar{K}2q.
	\end{align*}
Also, using \eqref{katoinequality}, we arrive at
	\begin{align*}
\left(\partial_t-\Delta\right)q&\le 2q\left(|\mathring{h}|^2+\frac{n-2}{2-\varepsilon n}|A^-|^2+\frac{1}{n}|H|^2-n\bar{K}\right).
	\end{align*}
Finally, for the evolution equation of $\frac{\left(\partial_t-\Delta\right)\frac{q}{-Q}}{\frac{q}{-Q}}$, we have the following computation
	\begin{align*}
\frac{\left(\partial_t-\Delta\right) \frac{q}{-Q}}{\frac{q}{-Q}} & =\frac{\left(\partial_t-\Delta\right) q}{q}-\frac{\left(\partial_t-\Delta\right) (-Q)}{-Q}-2\left\langle\nabla \log \frac{q}{-Q}, \nabla \log (-Q)\right\rangle \\
&<2\left(|\mathring{h}|^2+\frac{n-2}{2-\varepsilon n}|A^-|^2+\frac{1}{n}|H|^2-n\bar{K}\right)-2\left\langle\nabla \log \frac{q}{-Q}, \nabla \log (-Q)\right\rangle\\
&-2\left(|\mathring{h}|^2+\frac{n-2}{2-\varepsilon n}|A^-|^2+\frac{1}{n}|H|^2+\left(\frac{2}{\sqrt{n}}(1-2\varepsilon n)\right)n\bar{K}\right)\\
&<-2\left(2\sqrt{n}\bar{K}(1-2\varepsilon n)\right)-2\left\langle\nabla \log \frac{q}{-Q}, \nabla \log (-Q)\right\rangle,
	\end{align*}
where $2\sqrt{n}\bar{K}(1-2\varepsilon n)>0$, for $0<\varepsilon\ll \frac{1}{n^{5 / 2}}$. In the case of $H=0$, in the same way, we have
 	\begin{align*}
&\left(\partial_t-\Delta\right)Q\le2Q\left(|\mathring{h}|^2+\frac{n-2}{2-\varepsilon n}|A^-|^2+\left(\frac{2}{\sqrt{n}}(1-2\varepsilon n)\right)n\bar{K}\right),\\
&\left(\partial_t-\Delta\right)q\le 2q\left(|\mathring{h}|^2+\frac{n-2}{2-\varepsilon n}|A^-|^2-n\bar{K}\right).
	\end{align*}
Therefore, for the evolution equation of $\frac{\left(\partial_t-\Delta\right)\frac{q}{-Q}}{\frac{q}{-Q}}$, in the case where $H=0$, we have the following computation
	\begin{align*}
\frac{\left(\partial_t-\Delta\right) \frac{q}{-Q}}{\frac{q}{-Q}}
\le-2\left(2\sqrt{n}\bar{K}(1-2\varepsilon n)\right)-2\left\langle\nabla \log \frac{q}{-Q}, \nabla \log (-Q)\right\rangle,
	\end{align*}
which is the same evolution equation as in the case $H\neq0$. Hence, by the strong maximum principle there exists this constant $\mathcal{C}$ depending upon $\varepsilon,n$ and $\bar{K}$, such that
	\begin{align*}
	\frac{|\mathring{A}|^2}{f}=\frac{q}{-Q}\le\mathcal{C}e^{-2\left(2\sqrt{n}\bar{K}(1-2\varepsilon n)\right)t},
	\end{align*}
which completes the proof.
\end{proof}
Proposition \ref{theoreminfinitetime} implies, that there exists $\tau=\tau(n, \varepsilon)$, such that the inequality $|A|^2-a(|H|^2)+\varepsilon\left( \frac{|H|^2}{n-2}+4 n \sqrt{n} \bar{K}\right)<0$ holds for every $t \in(0, T) \cap\left[\tau \left(2\sqrt{n}\bar{K}(1-2\varepsilon n)\right)^{-1}, +\infty\right)$ on any solution initially satisfying \eqref{pinching}. 
If $T>\tau \left(2\sqrt{n}\bar{K}(1-2\varepsilon n)\right)^{-1}$, this means that at that time the solution satisfies the hypotheses of Theorem 1.1 in \cite{DongPu}. Consequently, the solution either exists forever and converges to a totally geodesic submanifold as $t \rightarrow \infty$, or else contracts to codimension one solution in finite time, from Theorem \ref{blowuptheorem}.

\end{document}